%% file: persistence.tex
\documentclass[12pt, a4paper,twoside]{article}

\usepackage[latin1]{inputenc}
\usepackage[english]{babel}

\usepackage{amsmath,amsthm,amsfonts,amssymb}
\usepackage{enumerate,enumitem}

\usepackage{pgf,tikz}
\usetikzlibrary{arrows}

\usepackage{vmargin}
\setmarginsrb{3cm}{2.5cm}{3cm}{3cm}{1cm}{1cm}{1cm}{1cm}

\usepackage{fancyhdr}
\newtheorem{lemma}{Lemma}
\newtheorem{proposition}{Proposition}
\newtheorem{theorem}{Theorem}
\newtheorem{remarkk}{Remark}

\DeclareMathOperator{\dd}{d\!}
\DeclareMathOperator{\ddd}{d^2\!}
\DeclareMathOperator{\supp}{supp}
\DeclareMathOperator{\diag}{diag}
\DeclareMathOperator{\Id}{Id}

\begin{document}

\title{Persistence of generalized roll-waves under viscous perturbation}
\author{Val\'erie Le Blanc\footnote{\noindent Universit\'e de Lyon; Universit\'e Lyon 1; INSA de Lyon, F-69621; \'Ecole Centrale de Lyon; CNRS, UMR5208, Institut Camille Jordan; 43, boulevard du 11 novembre 1918; F-69622 Villeurbanne-Cedex, France. E-mail address: leblanc@math.univ-lyon1.fr.}}
\maketitle

\begin{abstract}
  The purpose of this article is to study the persistence of solution of a hyperbolic system under small viscous perturbation. Here, the solution of the hyperbolic system is supposed to be periodic: it is a periodic perturbation of a roll-wave. So, it has an infinity of shocks. The proof of the persistence is based on an expansion of the viscous solution and estimates on Green's functions.
\end{abstract}


\section{Introduction}

In this paper, we consider a one-dimensional system
\begin{equation}\label{eqpara}
  u^\varepsilon_t+f(u^\varepsilon)_x=g(u^\varepsilon)+\varepsilon u^\varepsilon_{xx}
\end{equation}
with a smooth flux $f:\mathbb R^n\to\mathbb R^n$ and a smooth function $g:\mathbb R^n\to\mathbb R^n$. We assume that the corresponding system without viscosity
\begin{equation}\label{eqhyp}
  u_t+f(u)_x=g(u)
\end{equation}
is strictly hyperbolic.

We consider a piecewise smooth function $u$ which is a distributional solution of~\eqref{eqhyp} on the domain $\mathbb R\times[0;T^*]$. We assume that $u$ is periodic in the $x$ variable, with a period $L$ and that $u$ has $m$ noninteracting Lax shocks per period.

We show that $u$ is a strong limit of solutions $u^\varepsilon$ of~\eqref{eqpara} as $\varepsilon\to0$. This work is of course motivated by the conjecture that the admissible solutions of~\eqref{eqhyp} are strong limits of solutions of~\eqref{eqpara} with the same initial data.

In the case of scalar conservation laws, the proof of this conjecture uses the maximum principle~\cite{Volpert}, and in the case of special 2$\times2$ systems, R.~J.~DiPerna proved it by a compensated compactness argument~\cite{DiPerna}. For the general case of shocks, there is a first paper of J.~Goodman and Z.~P.~Xin which proves this conjecture for small amplitude Lax shocks~\cite{GoodmanXin}. This conjecture is also proved for a single non-characteristic Lax shock or overcompressive shock by F.~Rousset~\cite{Rousset}. Here, we only consider Lax shocks but we have an infinity of shocks.
        
An other motivation of this work states in the study of roll-waves, in fluid mechanics or in general hyperbolic systems with source terms. Indeed, P.~Noble proved the existence of roll-waves for this kind of system under assumptions on the source term \cite{Noble07}. Specifically, in the case of inviscid Saint Venant equations
\begin{equation*}
  \begin{cases}
    h_{t}+(hu)_{x}=0,\\
    \displaystyle(hu)_{t}+(g\cos\theta\frac{h^2}2+hu^2)_{x}=gh\sin\theta-c_{f}u^2,
  \end{cases}
\end{equation*}
one can prove that there exist roll-waves which are persistent under small perturbation~\cite{Noble}. So, there exist solutions of inviscid Saint Venant equations, near roll-waves. Here, the idea is to prove that there exists a family of solutions of the viscous Saint Venant system
\begin{equation}
  \label{StVenant}
  \begin{cases}
    h_{t}+(hu)_{x}=0,\\
    \displaystyle(hu)_{t}+(g\cos\theta\frac{h^2}2+hu^2)_{x}=gh\sin\theta-c_{f}u^2+\varepsilon(hu_{x})_{x}, 
  \end{cases}
\end{equation}
which tends to a solution of inviscid system as $\varepsilon$ goes to 0. We prove this result in the case of full viscosity.

We can now give the full set of assumptions and formulate our main result. First, we suppose that 

\begin{description}
\item[{\bf(H1)}]system~\eqref{eqhyp} is strictly hyperbolic.
\end{description}
That means that there exist smooth matrices $P(u), D(u)$ such that $$\dd f(u)=P(u)D(u)P(u)^{-1}$$ where $D(u)=\diag(\lambda_1(u),\dots,\lambda_n(u))$ is a diagonal matrix and $\lambda_i\neq\lambda_j$ for all $i\neq j$.

\begin{description}
\item[{\bf(H2)}]$u$ is a distributional solution of~\eqref{eqhyp} on $[0;T^*]$. Moreover, we suppose that $u$ is piecewise smooth, periodic, and has $m$ noninteracting and non-characteristic Lax shocks per period.
\end{description}
That means that $u$ is smooth except at the points $(x,t)$ of smooth curves $x=X_j(t)+iL, j=1,\dots,m, i\in\mathbb Z$ and that for all $j,k,t,|X_j(t)-X_k(t)|>2r>0$ (see Figure~\ref{figu}). 
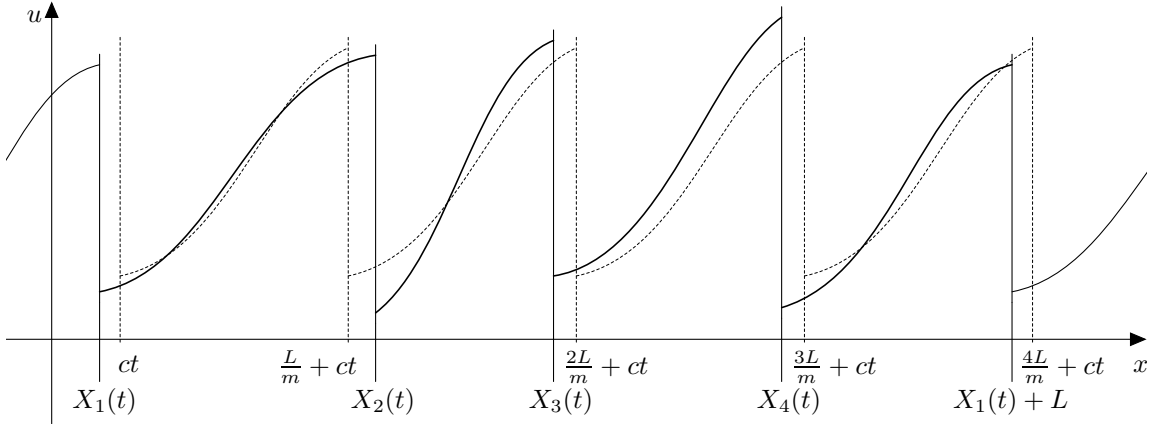
\begin{figure}[ht]
  \begin{center}
    \fontsize{10}{12}\selectfont
    \input{u.tex}
    \caption[Allure of solution $u$ over one period.]{Allure of solution $u$ over one period when $u$ is scalar and $m=4$. The periodic roll-wave is drawn in dotted line. The solution which checks our assumptions is represented by continuous line. One also placed the shocks for the two solutions.}
    \label{figu}
  \end{center}
\end{figure}

Moreover, following limits are finite:
$$\partial_x^ku^{j\pm}(t):=\partial_x^ku(X_j(t)\pm0,t) = \lim_{x\to X_j(t)^\pm}\partial_x^ku(x,t).$$
Since the shocks are non-characteristic $k$-Lax shocks, we have:
$$\lambda_1(u^{j-})\leq\dots\leq\lambda_{k-1}(u^{j-})<X_j'(t)<\lambda_{k}(u^{j-})\leq\dots\leq\lambda_n(u^{j-}),$$
$$\lambda_1(u^{j+})\leq\dots\leq\lambda_{k}(u^{j+})<X_j'(t)<\lambda_{k+1}(u^{j+})\leq\dots\leq\lambda_n(u^{j+}).$$
This assumption ensures the existence of at least one sonic point between two shocks.

We refer to~\cite{Noble} for the existence of such a solution in the case of Saint Venant equations. This result can be extended to general hyperbolic systems.

These assumptions imply that there exists a viscous profile for each shock. More precisely, for all $j$, there exists $V^j$ such that 
\begin{equation}
  \label{viscousprofile}
  V_{\xi\xi}^j-(f(V^j)-X_j'V^j)_\xi =0          
\end{equation}
and 
$$V^j(\pm\infty,t)=u^{j\pm}(t).$$ 
We will give more details on the properties of $V^j$ in Section~\ref{VISCOUSPROFILE}. Now, we only need to expose some assumption of linear stability. Consider for $\tau\leq T^*$, the operator
$$\mathcal L_\tau^j w = w_{zz}-(\dd f(V^j(z,\tau))-X_j'(\tau))w_z.$$
We assume that the viscous shock profiles are linearly stable. This assumption is equivalent to an Evans function criterion~\cite{ZumbrunHoward}.
\begin{description}
\item[{\bf(H3)}]$\forall \tau\in[0;T^*],j=1,\dots,m,\mathcal L_\tau^j$ is such that $D_\tau^j(\lambda)\neq0\ \ \forall\lambda,\Re\lambda\geq0,\lambda\neq0,$ and ${D_\tau^j}'(0)\neq0$, where $D_\tau^j$ is the Evans function of $\mathcal L_\tau^j$.
\end{description}

We can now state our main theorem:
\begin{theorem}\label{thethmpersistence}
  Under assumptions {\bf(H1)--(H2)--(H3)}, for all $\varepsilon>0$, there exists a unique solution $u^\varepsilon$ of~\eqref{eqpara} on $[0;T^*]$ such that
\begin{equation}\label{CI}u^\varepsilon(t=0,x)=u(t=0,x).\end{equation}
Moreover, we have the convergences $$\|u^\varepsilon-u\|_{L^\infty([0;T^*],L^1(0;L))}\to0,\quad\text{ as }\varepsilon \to 0.$$ And for any $\eta\in(0,1),$ $$\sup_{0\leq t\leq T^*,|x-X_j(t)|\geq\varepsilon^\eta}|u^\varepsilon(x,t)-u(x,t)|\to 0,\quad\text{ as }\varepsilon \to 0.$$
\end{theorem}

The proof of this theorem is done in three steps: construction of an ap\-prox\-i\-mate solution (which gives us an expansion of $u^{\varepsilon}$ in $\varepsilon$), estimates on the semigroup gen\-er\-at\-ed by linearized operator around this approximate solution and a Banach fixed point argument to deal with the full problem. \newline

The paper is organized as follows. In Section~\ref{CONSTRUCTION}, we build an approximate solution $u_{app}^\varepsilon$  of the full problem~\eqref{eqpara} close to $u$, solution of~\eqref{eqhyp} up to order 2 with respect to $\varepsilon$. This is done separating slow parts where $u_{app}^\varepsilon$ is close to $u$ and shock parts where ${u_{app}^\varepsilon}|_{X_j\pm\varepsilon^\gamma}$ is close to $V_j$. More precisely, one expands $u^\varepsilon$ in the slow part as 
$$u^\varepsilon(x,t)=u(x,t)+\varepsilon u_1(x,t)+\varepsilon^2 u_2(x,t)+o(\varepsilon^2)$$
where $u$ is the solution of~\eqref{eqhyp} and $u_i$ are solutions of the linearized equation of~\eqref{eqhyp} around $u$, which is well-posed thanks to assumption {\bf(H3)}. In shock parts, the expansion at shock $j$ is
$$u^\varepsilon(x,t)=V^j(\xi^j(x,t,\varepsilon),t)+\varepsilon V_1^j(\xi^j(x,t,\varepsilon),t)+\varepsilon^2V_2^j(\xi^j(x,t,\varepsilon),t)+o(\varepsilon^2)$$
where the stretched variable is  $\xi^j(x,t,\varepsilon)=\frac{x-X_j(t)}{\varepsilon}+\delta^j(t)$, $V^j$ is solution of viscous equation~\eqref{viscousprofile} and $V_i^j,i=1,2,$ are solutions of the linearized equation of~\eqref{viscousprofile} around $V^j$. Moreover, the functions $u_i,V_i^j$ are related by matching conditions which ensure regularity on the approximate solution $u_{app}^\varepsilon$, built by convex combination of the expansions:
$$u_{app}^\varepsilon= \sum_{j=1}^{m}\mu\left(\frac{x-X_j(t)}{\varepsilon^{\gamma}}\right)I^{j\varepsilon}(x,t)+\left(1-\sum_{j=1}^m  \mu\left(\frac{x-X_j(t)}{\varepsilon^{\gamma}}\right)\right)O^\varepsilon(x,t).$$
where $$\mu(x)=\left\{\begin{array}{l}0\text{ if }|x|>2,\\1\text{ if }|x|<1,\end{array}\right.$$ 
\begin{align*}
  &I^{j\varepsilon}(x,t)=V^j(\xi^j(x,t,\varepsilon),t)+\varepsilon V_1^j(\xi^j(x,t,\varepsilon),t)+\varepsilon^2V_2^j(\xi^j(x,t,\varepsilon),t),\\
  &O^\varepsilon(x,t)=u(x,t)+\varepsilon u_1(x,t)+\varepsilon^2 u_2(x,t).
\end{align*}
With this construction, we prove the theorem:
\begin{theorem}\label{thmuapp}
  There exists an approximate solution $u_{app}^\varepsilon$ of~\eqref{eqpara} defined on $[0;T^*]$. If $\varphi$ is a smooth change of variable which fixes the shocks ($\forall t,i,j, \varphi((j-1)\frac Lm+iL,t)=X_j(t)+iL$), and $\tilde u_{app}^\varepsilon(z,t)=u_{app}^\varepsilon(\varphi(z,t),t)$, then $\tilde u_{app}^\varepsilon$ verifies the equation
  $$(\tilde u_{app}^\varepsilon)_{t}+f(\tilde u_{app}^\varepsilon)_{x}-\varepsilon(\tilde u_{app}^\varepsilon)_{xx}-g(\tilde u_{app}^\varepsilon)=\tilde q^\varepsilon(x,t)$$
  with the following estimates on $\tilde q^\varepsilon$
  \begin{eqnarray}
    \label{Linfinity}\|\tilde q^\varepsilon\|_{L^\infty},\|\tilde q_t^\varepsilon\|_{L^\infty},\|\tilde q_{tt}^\varepsilon\|_{L^\infty}\leq C\varepsilon^{2\gamma},\\
    \label{L1}\|\tilde q^\varepsilon\|_{L^1(0;L)},\|\tilde q_t^\varepsilon\|_{L^1(0;L)},\|\tilde q_{tt}^\varepsilon\|_{L^1(0;L)}\leq C\varepsilon^{3\gamma},\\
    \label{qz}\|\tilde q_{z}^\varepsilon\|_{L^\infty},\|\tilde q_{zt}^\varepsilon\|_{L^\infty}\leq C\varepsilon^{\gamma},\|\tilde q_{z}^\varepsilon\|_{L^1(0;L)},\|\tilde q_{zt}^\varepsilon\|_{L^1(0;L)}\leq C\varepsilon^{2\gamma},\\
    \label{qzz}\|\tilde q_{zz}^\varepsilon\|_{L^\infty}\leq C,\|\tilde q_{zz}^\varepsilon\|_{L^1(0;L)}\leq C\varepsilon^{\gamma}.
  \end{eqnarray}
\end{theorem}
Here, $u_{app}^\varepsilon$ is constructed as a perturbation of $u$ going to order 2, which allows us to have estimates on $\tilde q_{zz}^\varepsilon$ in $L^1$ in $\varepsilon^\gamma$. This property will be useful to prove the convergence of $u_{app}^\varepsilon-u^\varepsilon$ to 0.

In Section~\ref{GREEN}, we linearize~\eqref{eqpara} in the neighbourhood of the approximate solution $u_{app}^\varepsilon$ and we compute estimates on the Green's function. To do so, we use the method of iterative construction of the Green's function, which was first used by E.~Grenier and F.~Rousset in~\cite{GrenierRousset}. So, we consider approximations of the Green's functions in neighbourhood of the shocks (given by K.~Zumbrun and P.~Howard in~\cite{ZumbrunHoward}) and we build our own approximation far away from the shock, using the characteristic curves.

The last section is dedicated to the proof of theorem
\begin{theorem}\label{thmconvergence}
  Under assumptions {\bf(H1)--(H2)--(H3)}, for all $\varepsilon$, there exists $u^\varepsilon$ solution of~\eqref{eqpara}-\eqref{CI} on $(0;T^*)$. And this $u^\varepsilon$ verifies the convergences:
  $$\|u^\varepsilon-u_{app}^\varepsilon\|_{L^\infty((0;T^*)\times\mathbb R)}\to 0,$$
  $$\|u^\varepsilon-u_{app}^\varepsilon\|_{L^\infty((0;T^*),L^1(\mathbb R))}\to 0$$
  when $\varepsilon$ goes to zero.
\end{theorem}
This is done using standard arguments for parabolic problems. Indeed, we combine estimates on $\tilde q^\varepsilon$, and estimate on the Green's function to obtain estimates on $u^\varepsilon-u_{app}^\varepsilon$, and its derivatives, depending on $\varepsilon$ and uniform in time for $\varepsilon$ small enough. Then, using the convergence of $u_{app}^\varepsilon$ to $u$, we immediately deduce Theorem~\ref{thethmpersistence}.

\section{Construction of the approximate solution}\label{CONSTRUCTION}
The purpose of this section is to prove Theorem~\ref{thmuapp} on the existence of the approximate solution $u_{app}^\varepsilon$ of~\eqref{eqpara}. In a first step, we compute formally this approximate solution using outer and inner expansions of order 2. Indeed, in slow part, where $\nabla u$ is bounded, the solution $u^\varepsilon$ of~\eqref{eqpara} may be approximated by truncation of the formal series
$$u^\varepsilon(x,t) \sim O^\varepsilon(x,t)=u(x,t)+\varepsilon u_1(x,t)+\varepsilon^2 u_2(x,t)$$ 
where $u$ is the solution of~\eqref{eqhyp} we want to approach. Similarly, near the shocks $j$, we search for $u_{app}^\varepsilon$ with the inner expansion 
$$I^{j\varepsilon}(x,t)=V^j(\xi^j(x,t,\varepsilon),t)+\varepsilon V_1^j(\xi^j(x,t,\varepsilon),t)+\varepsilon^2V_2^j(\xi^j(x,t,\varepsilon),t)$$ where $\xi^j(x,t,\varepsilon)=\frac{x-X_j(t)}{\varepsilon}+\delta_0^j(t)+\varepsilon\delta_1^j(t)$ is the stretched variable and $V^j$ is the viscous shock profile, solution of~\eqref{viscousprofile}.
We match this expansion by continuity of $u_{app}^\varepsilon$ and its spatial derivatives.

In this section, we formally substitute these expansions in~\eqref{eqpara} to find equations satisfied by $u_i$ and $V_i^j$, $i=1,2,j=1,\dots,m$, and matching conditions. Then, we prove the existence of the $u_i$ and $V_i^j$ on $(0;T^*)$. Furthermore, we give rigorous estimates on the error terms. We can remark here that we search for an approximation of order 2, this will be useful to obtain good estimates on the second derivatives of the error term.

\subsection{Formal calculation and derivation of the equations}

Substituting $O^\varepsilon$ into~\eqref{eqpara} and identifying the power of $\varepsilon$ in the expressions, we get for $x\neq X_j(t)$: 
\begin{alignat*}{1}
  \mathcal O(\varepsilon^0):&\ u_{t}+(f(u))_{x}-g(u)=0,\\
  \mathcal O(\varepsilon^1):&\ u_{1,t}+(\dd f(u)\cdot u_{1})_{x}-\dd g(u)\cdot u_{1}=u_{xx},\\
  \mathcal O(\varepsilon^2):&\ u_{2,t}+(\dd f(u)\cdot u_{2})_{x}-\dd g(u)\cdot u_{2}=u_{1xx}-\frac12(\ddd f(u)\cdot (u_{1},u_{1}))_{x}\\
  &\hspace{9cm}+\frac12\ddd g(u)\cdot (u_{1},u_{1}).
\end{alignat*}
A similar calculation for $I^{j\varepsilon}$ yields the set of equations:
\begin{alignat*}{1}
  \mathcal O(\varepsilon^{-1}):&\ V_{\xi\xi}^j-(f(V^j)-X_j'V^j)_{\xi}=0,\\
  \mathcal O(\varepsilon^0):&\ V_{1\xi\xi}^j-((\dd f(V^j)-X_j')\cdot V_{1}^j)_{\xi}=V_{t}^j+V_{\xi}^j\delta_{0t}^j-g(V^j),\\
  \mathcal O(\varepsilon^1):&\ V_{2\xi\xi}^j-((\dd f(V^j)-X_j')\!\cdot\!V_{2}^j)_{\xi}=V_{1t}^j+V_{1\xi}^j\delta_{0t}^j+V_{\xi}^j\delta_{1t}^j+\frac12(\ddd f(V^j)\!\cdot\!(V_{1}^j,V_{1}^j))_{\xi}\\
  &\hspace{10cm}-\dd g(V^j)\cdot V_{1}^j.
\end{alignat*}

We remark that the equations for $u_i$ are hyperbolic equations: the first one is~\eqref{eqhyp}, so nonlinear, and the others are the linearization of~\eqref{eqhyp} around $u$. Similarly, the equations for the shock profiles are ordinary equations: nonlinear for $V^j$, we recognize~\eqref{viscousprofile}, and its linearization around $V^j$ for $V_1^j$ and $V_2^j$. To maximize the order of the approximation, we couple these equations with boundary conditions, connecting $u_i$ and $V_i^j$. First, we note
$$\partial_x^ku_i^{j\pm}(t):=\partial_x^ku_i(X_j(t)\pm0,t) = \lim_{x\to X_j(t)^\pm}\partial_x^ku_i(x,t),\quad i=1,2.$$
Then, we rewrite $O^\varepsilon$ and $I^\varepsilon$ with the variable $\xi$, in a vicinity of shock $j$, and we ask the two functions to coincide as $\varepsilon$ goes to $0$. Therefore, we make Taylor expansion of order 2 with respect to $\varepsilon$. For example, for $\xi>0$, large enough,
\begin{equation*}
  \begin{split}
     O^\varepsilon(X_j(t)+\varepsilon(\xi&-\delta_0^j(t)-\varepsilon\delta_1^j(t)),t)=u^{j+}(t)+\varepsilon \big(u_1^{j+}(t)+u_x^{j+}(\xi-\delta_0^j)\big)\\
     &+\frac{\varepsilon^2}2\big(u_2^{j+}(t)+2u_{1x}^{j+}(t)(\xi-\delta_0^j)-2u_x^{j+}\delta_1^j+u_{xx}^{j+}(t)(\xi-\delta_0^j)^2\big)+o(\varepsilon^2)
  \end{split}
\end{equation*}
and
$$I^\varepsilon(X_j(t)+\varepsilon(\xi-\delta_0^j(t)-\varepsilon\delta_1^j(t)),t)=V^j(\xi,t)+\varepsilon V_1^j(\xi,t)+\varepsilon^2V_2^j(\xi,t)+o(\varepsilon^2)$$
Identifying the terms of same order on $\varepsilon$, we get as $\xi\to\pm\infty$: 
\begin{alignat}{1}
  \label{match0}V^j(\pm\infty,t)&=u^{j\pm}(t),\\
  \label{match1}V_{1}^j(\xi,t)&=u_{1}^{j\pm}(t)+u_{x}^{j\pm}(t)(\xi-\delta_{0}^j(t))+o(1),\\
  \label{match2}V_{2}^j(\xi,t)&=u_{2}^{j\pm}(t)+u_{1x}^{j\pm}(t)(\xi-\delta_{0}^j(t))+\frac12u_{xx}^{j\pm}(t)(\xi-\delta_{0}^j(t))^2-u_{x}^{j\pm}(t)\delta_{1}^j(t)+o(1).
\end{alignat}
For more details on the computation of these conditions, we refer to~\cite{Fife}.

\subsection{Existence of solutions of the outer and inner problems}
In this section, we show that the solutions $u_i$ and $V_i^j$ of the previous equations exist under assumption (\textbf{H3}) on the spectral stability of the viscous shock profile. We first remark that the leading-order outer function $u$ is exactly the solution of~\eqref{eqhyp} which we want to approximate. Therefore, we first prove the existence of the $V^j$, and then we prove the existence of $u_1$ and all the $V_1^j$. Simultaneously, we prove the existence of the $\delta_0^j$. Similarly, we prove the existence of $u_2, V_2^j$ and $\delta_1^j$.

\subsubsection{Construction at order 0}\label{VISCOUSPROFILE}
In this section, we deal with the existence of $u,V^j$ which satisfy equations~\eqref{eqhyp}-\eqref{viscousprofile} and matching condition~\eqref{match0}. The existence of $u$ is exactly assumption (\textbf{H2}). Since $u$ is a distributional solution, $u$ verifies Rankine-Hugoniot conditions at each shock $j$:
$$f(u^{j+})-f(u^{j-})=X_j'(t)(u^{j+}-u^{j-}).$$

This property ensures the existence of the viscous shock profile $V^j$ which verifies
\begin{equation}\label{profile}V_{\xi\xi}^j-(f(V^j)-X_j'(t)V^j)_{\xi}=0\end{equation}
and the asymptotic conditions:
$$V^j(\pm\infty,t)=u^{j\pm}(t).$$
For the existence for all $t$ of such a profile and its properties, we refer to N.~Kopell and L.~N.~Howard~\cite{KopellHoward}.
In this section, we just recall the convergence rate of the profile and its derivatives as $\xi\to\pm\infty$. 
Since $u^{j+}$ and $u^{j-}$ are hyperbolic rest points for the ordinary differential equation~\eqref{profile}, we have for some $\omega>0$ and for any $\alpha\in\mathbb{N}$, 
\begin{eqnarray}
  \label{estprofile1}|\partial_t^\alpha V^j(\xi,t)-\partial_t^\alpha u^{j\pm}(t)|&\leq &e^{-\omega|\xi|},\quad\forall\xi\in \mathbb R,\\
  \label{estprofile2}|\partial_\xi^\alpha V^j(\xi,t)|&\leq &e^{-\omega|\xi|}, \quad\forall \xi\in\mathbb R.
\end{eqnarray}

\subsubsection{Construction at order 1: existence of $V_1^j$ and $u_1$}
In this section, we prove the existence of $u_1, V_1^j,\delta_0^j$ on $(0;T^*)$ such that 
\begin{eqnarray}\label{u1}~\!\!&~\!\!&\!\!\!\!u_{1,t}+(df(u)\cdot u_{1})_{x}-dg(u)\cdot u_{1}=u_{xx},\\
  ~\!\!&~\!\!&\!\!\!\!V_{1\xi\xi}^j-((df(V^j)-X_j')\cdot V_{1}^j)_{\xi}=V_{t}^j+V_{\xi}^j\delta_{0t}^j-g(V^j),\\
  ~\!\!&~\!\!&\!\!\!\!V_{1}^j(\xi,t)=u_{1}^{j\pm}(t)+u_{x}^{j\pm}(t)(\xi-\delta_{0}^j(t))+o(1),\xi\to\pm\infty.\end{eqnarray}
We first remark that these equations are linear. As in~\cite{Rousset}, it is convenient to deal with bounded solutions. Therefore, we write 
$$U_1^j=V_1^j-D_1^j$$
where $D_1^j$ is a smooth function such that:
$$D_1^j=\left\{\begin{array}{ll}\xi u_x^{j-}(t)&\text{ if }\xi<-1,\\\xi u_x^{j+}(t)&\text{ if }\xi>1.\end{array}\right.$$
Consequently, $U_1^j$ solves:
\begin{eqnarray}
  \label{U1}~\!\!&~\!\!&\!\!\!\!U_{1\xi\xi}^j-((df(V^j)-X_j')\cdot U_1^j)_\xi=\delta_{0t}^jV_\xi^j+h^j(\xi,t),\\
  \label{U1bord}~\!\!&~\!\!&\!\!\!\! U_1^j(\pm\infty,t)=u_{1}^{j\pm}(t)-\delta_{0}^j(t)u_{x}^{j\pm}(t)
\end{eqnarray}	
with $$h^j(\xi,t)=-D_{1\xi\xi}^j+V_t^j+((\dd f(V^j)-X'_j)D_1^j)_\xi-g(V^j).$$

From estimates~\eqref{estprofile1},~\eqref{estprofile2}, we deduce that $h$ satisfies: $$h^j(\xi,t)=\frac \dd{\dd t}u^{j\pm}(t)+(\dd f(u^{j\pm})-X_j')u_x^{j\pm}(t)-g(u^{j\pm})+O(e^{-\alpha|\xi|}),\alpha>0.$$
And, since $u$ is a smooth solution of~\eqref{eqhyp}, we have $h\in L^1(\mathbb R)$ and: $h^j(\xi,t)=O(e^{-\alpha|\xi|}).$ Integrating~(\ref{U1}) with respect to $\xi$ yields
\begin{equation}\label{U1int}U_{1\xi}^j-(\dd f(V^j)-X_j')U_1^j=\delta_{0t}^jV^j+\int_0^\xi h^j(\eta,t)\dd\eta+C^j(t)\end{equation}
where $C^j(t)$ is a constant, only depending on $t$.

Let us solve the problem~\eqref{u1}-\eqref{U1int} with matching condition~\eqref{U1bord}. Following~\cite{Rousset}, we construct the solution of this system in two steps. First, for all $j$, we fix $t$ and $\delta_0^j$, we find $U_1^j$ solution of~\eqref{U1int} with finite limits at $\pm\infty$. Since these limits are explicit and only depends on $t, \delta_0^j$, and $C^j$, we use the matching condition \eqref {U1bord} to rewrite~\eqref{u1} as a hyperbolic boundary value problem where $u_1$ and $\delta_0^j$ are the only unknowns. After solving this system, we use the previous construction to obtain $U_1^j$ solution of~\eqref{U1int} with matching conditions~\eqref{U1bord}.

So, we fix $t$ and $\delta_0^j$ for all $j$. With exactly the same arguments as in~\cite{Rousset}, we prove the existence of $U_1^j$ for all $j$. Hence, using assumption (\textbf{H3}) on the viscous shock profile and theory of Fredholm operators, we show that $U_1^j$ exists and the limits satisfy:
$$\lim_{\xi\to\pm\infty}U_1^j(\xi,t)=-(\dd f(u^{j\pm})-X_j')^{-1}(\delta_{0t}^ju^{j\pm}+H^{j\pm}+C^j)$$
where $H^{j\pm}=\int_0^{\pm\infty}h(\eta,t)\dd\eta$.

We now use matching conditions~(\ref{U1bord}) to eliminate $C^j$ in these relations. Indeed, we have
$$(\dd f(u^{j+})-X_j')(u_{1}^{j+}-\delta_{0}^ju_{x}^{j+})=-(\delta_{0t}^ju^{j+}+H^{j+}+C^j),$$
$$(\dd f(u^{j-})-X_j')(u_{1}^{j-}-\delta_{0}^ju_{x}^{j-})=-(\delta_{0t}^ju^{j-}+H^{j-}+C^j),$$
and their difference is
\begin{equation}
  \label{u1shock}A^{j+}u_1^{j+}-A^{j-}u_1^{j-}+\delta_{0t}^j(u^{j+}-u^{j-})=\delta_0^j(A^{j+}u_x^{j+}-A^{j-}u_x^{j-})-(H^{j+}-H^{j-})
\end{equation}
where $A^{j\pm}=\dd f(u^{j\pm})-X_j'(t)$.

Now, we have to solve~\eqref{u1},~\eqref{u1shock}. In order to find a solution of this system, we rewrite it by fixing the shocks. Since the shocks do not interact, we can define a change of variable $Z$ (see Figure~\ref{figZ}) which is bijective, continuous in $(x,t)$, piecewise linear in $x$ and piecewise smooth: $$Z(t,x)=\frac{x-X_j(t)}{X_{j+1}(t)-X_j(t)}\frac Lm+(j-1)\frac Lm\quad\text{ if }x\in[X_j(t);X_{j+1}(t)], i=1,\dots,m.$$
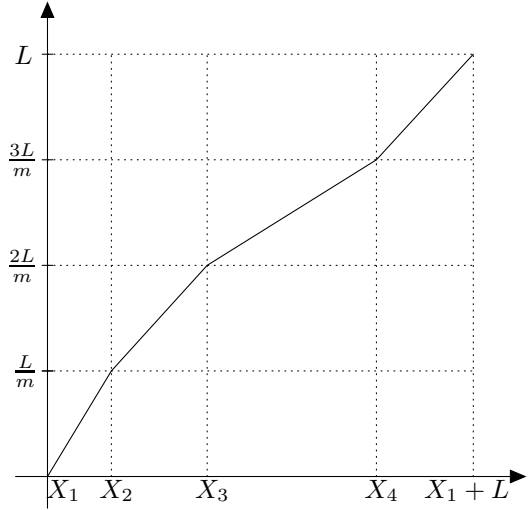
\begin{figure}[ht]
  \begin{center}
    \fontsize{10}{12}\selectfont
    \input{z.tex}
    \caption{Example of the change of variable $Z$ in the case $m=4$, for some t.}
    \label{figZ}
  \end{center}
\end{figure}\\
We also define $v_1$ by $$u_1(x,t)=v_1(Z(t,x),t).$$ 
It follows from these definitions that $v_1$ solves
\begin{equation}v_{1t}+(Z_x\dd f(u)+Z_t)(Z^{-1}(t,z),t)v_{1z}+\bar h(z,t)v_1-\tilde h(z,t)=0\end{equation}
where $z\mapsto x=Z^{-1}(t,z)$ is the inverse of $x\mapsto z=Z(t,x)$, and, $\bar h$ and $\tilde h$ only depend on $Z, u, u_{x},$ and $ u_{xx}$.

We now use the fact that $\dd f(u)$ is diagonalizable, $\dd f(u)=P(u)^{-1}D(u)P(u)$ so 
$$(Z_x\dd f(u)+Z_t)(Z^{-1}(t,z),t)=\tilde P(z,t)^{-1}\tilde D(z,t)\tilde P(z,t).$$ 
Since zeroth order terms do not play any role in the wellposedness issue, we consider the simplified system
\begin{eqnarray}
  \label{wtransport} w_{1t}+\tilde D(z,t)w_{1x}=\tilde k(z,t),\\
  \label{wchoc} A^{j+}(P^{j+})^{-1}w_1^{j+}-A^{j-}(P^{j-})^{-1}w_1^{j-}+\delta_{0t}^j(u^{j+}-u^{j-})=l^{j}(t)
\end{eqnarray}	
with $\tilde k$ and $l^{j}$ known functions. Therefore, we have to solve this system on $[0;L]$ under periodic boundary conditions. Since $0$ and $L$ correspond to the same shock, the periodic boundary conditions are in fact the shock conditions~(\ref{wchoc}) for $j=1$. 

Equation~\eqref{wtransport} is a linear transport equation on $w_{1i}, i=1,\dots, n$. Since, generically, the existence of $u$ smooth on $[0;T^*]$ ensures that the characteristics do not intersect on $[0;T^*]$, they can be used to build $w_{1,i}$ smooth between shocks, using the initial condition. So, it suffices to verify that the conditions~(\ref{wchoc}) at the shocks are well-posed. We must therefore count the incoming and outgoing information at the shock. As we can see in Figure \ref{figcaracteristics},  for $i<k$, the incoming characteristics come from the right, so we obtain the value of $w_{1i}^{j+}.$ For $i>k$, the incoming characteristics come from the left, so we get $w_{1i}^{j-}.$ And for $i=k$, the sign of the eigenvalue change between two shocks: negative on the right of a shock and positive on the left, so, using again characteristic construction, $w_{1k}$ is defined on the whole interval delimited by the shocks: we obtain $w_{1k}^{j+}$ and $w_{1k}^{j-}$.By this method we have built $w_{1i}$ on the right or left side of each shock. We now use the boundary conditions~\eqref{wchoc} to obtain all the components of $w_1^j$, and $\delta_{0t}^{j}$.
\begin{figure}[ht]
  \begin{center}
    \fontsize{10}{12}\selectfont
    \input{w1.tex}
    \caption[Characteristic curves between two shocks.]{Characteristic curves between two shocks, example with $m=2$. We also plot $s_1$ and $s_2$ which are sonic points for $k$-th eigenvalue, that means $\lambda_k(s_j)=0$.}
    \label{figcaracteristics}
  \end{center}
\end{figure}
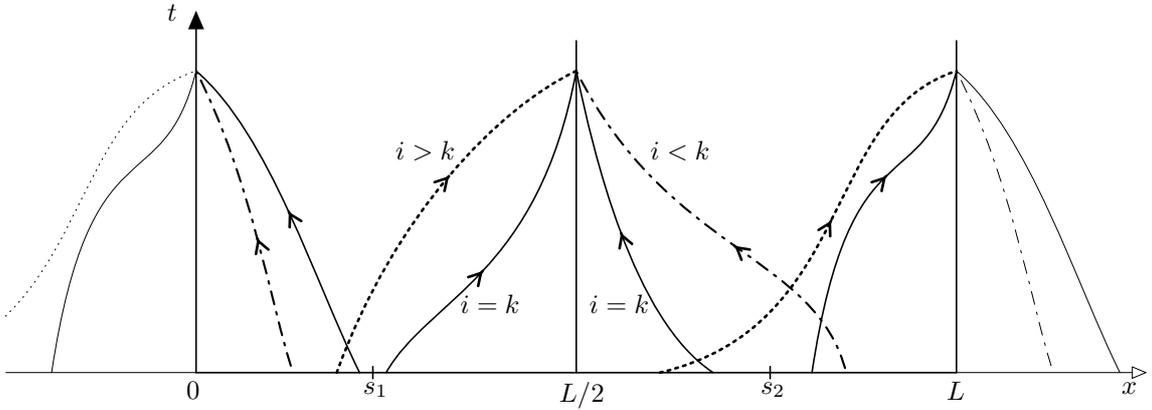

Indeed, if we note by $r_{i}$ the i-th eigenvector of $\dd f(u)$, and $w_{1}=\tilde P\sum_{i}a_{i}r_{i}$, then $a_{i}^{j+}$ is known for $i\leq k$ and $a_{i}^{j-}$ for $i\geq k$ by our construction and we rewrite~\eqref{wchoc} as
$$\sum_{i=1}^n A^{j+}a_{i}^{j+}r_{i}^{j+}-\sum_{i=1}^n A^{j-}a_{i}^{j-}r_{i}^{j-}+\delta_{0t}^{j}(u^{j+}-u^{j-})=l^j(t)$$
or equivalently
$$\sum_{i=1}^n (\lambda_{i}^{j+}-X_{j}')a_{i}^{j+}r_{i}^{j+}-\sum_{i=1}^n (\lambda_{i}^{j-}-X_{j}')a_{i}^{j-}r_{i}^{j-}+\delta_{0t}^{j}(u^{j+}-u^{j-})=l^j(t).$$
This yields the linear system on the unknowns $a_i^{j+}$ for $i>k$, $a_i^{j-}$ for $i<k$, and $\delta_{0t}^j$ 
\begin{multline*}
  \sum_{i>k} (\lambda_{i}^{j+}-X_{j}')a_{i}^{j+}r_{i}^{j+}-\sum_{i<k} (\lambda_{i}^{j-}-X_{j}')a_{i}^{j-}r_{i}^{j-}+\delta_{0t}^{j}(u^{j+}-u^{j-})\\
  =l^j(t)-\sum_{i\leq k} (\lambda_{i}^{j+}-X_{j}')a_{i}^{j+}r_{i}^{j+}+\sum_{i\geq k} (\lambda_{i}^{j-}-X_{j}')a_{i}^{j-}r_{i}^{j-}.    
\end{multline*}
These equations have unique solutions if and only if the system obtained is invertible for all $t$, that is the  Majda-Liu condition:
$$\forall j=1,\dots,m,\quad\det(r_1^{j-},\dots,r_{k-1}^{j-},u^{j+}-u^{j-},r_{k+1}^{j+},\dots,r_n^{j+})\neq0.$$
Using~\cite{ZumbrunSerre}, our assumption (\textbf{H3}) implies Majda-Liu condition. To finish the construction of the approximate solution, we use again the characteristics. By this way, we have built a solution on the whole space $\mathbb R$. To ensure the regularity of the solution far away from the shocks, we only need suitable compatibility conditions on the initial data.

Finally, we have proved the existence of $u_{1}$ and $\delta_{0t}^j$ for all $j$ and for $0\leq t\leq T^*$. The previous construction give us $V_{1}^j$ for all $j$. We can then apply the same method to obtain the existence of $V_{2}^{j}, \delta_{1t}^j$ and $u_{2}$, since the linear system has the same terms of maximal order.

\begin{remarkk}
\em  Since the construction of viscous shock profile only depends on the shock, we can use the previous construction even if $u$ is not periodic. However, we will see in the following that the periodicity of $u$ allows us first to obtain bounds on $u_1$ and secondly to build the Green's function in Section~\ref{GREEN}. 
\end{remarkk}

\subsection{Construction of the approximate solution}
We complete the construction of an approximate solution of equation~\eqref{eqpara}. First, we define a smooth function $\mu$ such that:
$$\mu(x)=\left\{\begin{array}{l}0\text{ if }|x|>2,\\1\text{ if }|x|<1.\end{array}\right.$$
Then, the approximate solution $u_{app}^\varepsilon$ is defined as
$$u_{app}^\varepsilon= \sum_{j=1}^{m}\mu\left(\frac{x-X_j(t)}{\varepsilon^{\gamma}}\right)I^{j\varepsilon}(x,t)+\left(1-\sum_{j=1}^m\mu\left(\frac{x-X_j(t)}{\varepsilon^{\gamma}}\right)\right)O^\varepsilon(x,t)$$
and $u_{app}^\varepsilon$ verifies
$$(u_{app}^{\varepsilon})_t+f(u_{app}^\varepsilon)_x-\varepsilon (u_{app}^{\varepsilon})_{xx}-g(u_{app}^\varepsilon)=q^\varepsilon$$
where $q^\varepsilon(x,t)=\sum_{i=1}^3q_i^\varepsilon(x,t)$ is an error term given by
\begin{equation*}
  \begin{split}
    q_1^\varepsilon(x,t)=
    &(1-\mu^j)\big[\big(f(O^\varepsilon)-f(u)-\varepsilon \dd f(u)\cdot 		u_1-\varepsilon^2\dd f(u)\!\cdot\! u_2-\frac{\varepsilon^2}2\ddd f(u)\!\cdot\!(u_1,u_1)\big)_x\\
    &-\big(g(O^\varepsilon)-g(u)-\varepsilon \dd g(u)\cdot u_1-\varepsilon^2\dd g(u)\cdot u_2-\frac{\varepsilon^2}2\ddd g(u)\cdot(u_1,u_1)\big)\\
    &-\varepsilon^3u_{2xx}\big],\\
    q_2^\varepsilon(x,t)=
    &\mu^j[(f(I^{j\varepsilon})\!-\!f(V^j)\!-\!\varepsilon \dd f(V^j)\cdot V_1^j-\varepsilon^2\dd f(V^j)\!\cdot\! V_2^j\!-\!\frac{\varepsilon^2}2\ddd f(V^j)\!\cdot\!(V_1^j,V_1^j))_x\\
    &-(g(I^{j\varepsilon})-g(V^j)-\varepsilon \dd g(V^j)\cdot V_1^j)\\
    &+\varepsilon^2(\delta_{1t}V_{1\xi}^j+V_{2t}^j+\delta'V_{2\xi}^j)],\\
    q_3^\varepsilon(x,t)=
    &\mu_{t}^j(I^{j\varepsilon}-O^\varepsilon)-\varepsilon  \mu_{xx}^j(I^{j\varepsilon}- O^\varepsilon) -2\varepsilon \mu_{x}^j(I^{j\varepsilon}-O^\varepsilon)_x +\mu_{x}^j(f(I^{j\varepsilon})-f(O^\varepsilon))\\
    &+f(\mu^jI^{j\varepsilon}+(1-\mu^j)O^\varepsilon)_x-(\mu^jf(I^{j\varepsilon})+(1-\mu^j)f(O^\varepsilon))_x\\
    &-g(\mu^jI^{j\varepsilon}+(1-\mu^j)O^\varepsilon)-(\mu^jg(I^{j\varepsilon})+(1-\mu^j)g(O^\varepsilon)),
  \end{split}
\end{equation*}
and $\mu^j=\mu\left(\frac{x-X_j(t)}{\varepsilon^\gamma}\right)$.

We now want to prove that $u_{app}^\varepsilon$ is a good approximation of $u^\varepsilon$, that means $u^\varepsilon-u_{app}^\varepsilon\to 0$ when $\varepsilon\to0$. Therefore, we define $w^\varepsilon=u^\varepsilon-u_{app}^\varepsilon$ which solves:
\begin{equation}\label{w}w_t+(\dd f(u_{app}^\varepsilon)\cdot w)_x-\varepsilon w_{xx}-\dd g(u_{app}^\varepsilon)\cdot w=-q^\varepsilon+Q_1(u_{app}^\varepsilon,w)-Q_2(u_{app}^\varepsilon,w)_x\end{equation}
with $$Q_1(u_{app}^\varepsilon,w)=g(w+u_{app}^\varepsilon)-g(u_{app}^\varepsilon)-\dd g(u_{app}^\varepsilon)\cdot w$$ and $$Q_2(u_{app}^\varepsilon,w)=f(w+u_{app}^\varepsilon)-f(u_{app}^\varepsilon)-\dd f(u_{app}^\varepsilon)\cdot w$$ which are at least quadratic terms in $w$.

\subsection{Estimates on the error term}\label{ESTIMATESQ}
To end with the proof of Theorem~\ref{thmuapp}, it remains to compute the estimates on the error term $q^\varepsilon$. As in~\cite{GoodmanXin}, we can estimate the support of functions $q_i^\varepsilon$:
\begin{eqnarray*}\supp(q_1^\varepsilon)\subset\{(x,t): |x-X_j(t)|\geq\varepsilon^\gamma\},\\
  \supp(q_2^\varepsilon)\subset\{(x,t): |x-X_j(t)|\leq2\varepsilon^\gamma\},\\
  \supp(q_3^\varepsilon)\subset\{(x,t): \varepsilon^\gamma\leq|x-X_j(t)|\leq2\varepsilon^\gamma\}.
\end{eqnarray*}

To obtain estimates on $q_i^\varepsilon$ and their derivatives, we first need to fix the shocks by a smooth change of variable. This manipulation cannot be avoided for the estimates on $q_{it}^\varepsilon, q_{itt}^\varepsilon$. So, we define 
\begin{equation}
  \label{phi}
  \varphi(z,t)=z+\sum_{j=1}^m\alpha_j(z,t)\left(X_j(t)-(j-1)\frac Lm-\varepsilon\delta^j(t)\right)
\end{equation}
where $\alpha_j(\cdot,t)$ are smooth functions, such that $\sum_j\alpha_j\equiv1$, $\varphi$ is increasing, $\varphi_z>0$ and $\alpha_j(\cdot,t)\equiv1$ on a neighbourhood $[(j-1)\frac Lm-r;(j-1)\frac Lm+r]$ of $(j-1)\frac Lm$. We recall that in assumption~(\textbf{H2}) we have supposed that $|X_{j+1}-X_j|>2r,$ which ensures the existence of such a $\varphi$.

With the notations $$\tilde w(z,t)=w(\varphi(z,t),t), \tilde u_{app}^\varepsilon(z,t)=u_{app}^\varepsilon(\varphi(z,t),t), \tilde q^\varepsilon (z,t)=q^\varepsilon(\varphi(z,t),t),$$ equation~\eqref{w} becomes:
\begin{multline}\label{tw}
  \tilde w_t+\frac1{\varphi_z}(\dd f(\tilde u_{app}^\varepsilon)\cdot \tilde w)_z-\left(\frac{\varphi_t}{\varphi_z}+\frac{\varepsilon \varphi_{zz}}{\varphi_z^3}\right)\tilde w_z-\varepsilon\left(\frac1{\varphi_z^2}\tilde w_{z}\right)_z-\dd g(\tilde u_{app}^\varepsilon)\cdot \tilde w\\
  =-\tilde q^\varepsilon+Q_1(\tilde u_{app}^\varepsilon,\tilde w)-\frac1{\varphi_z}Q_2(\tilde u_{app}^\varepsilon,\tilde w)_z.    
\end{multline}

We now prove estimates~\eqref{Linfinity}-\eqref{L1}-\eqref{qz}-\eqref{qzz}. Using the fact that $f$ and $g$ are smooth and $u_i$ is piecewise smooth, with discontinuities only at $x=X_j(t)$, we have the following estimates for $\tilde q_1^\varepsilon$:
$$\|\tilde q_1^\varepsilon\|_{L^\infty}\leq C\varepsilon^3$$ where $C$ is a constant which does not depend on $\varepsilon$. Integrating this inequality, we get:
$$\|\tilde q_1^\varepsilon\|_{L^1(0;L)}\leq C\varepsilon^3.$$
Moreover, if $\varepsilon$ is small enough (such that $2\varepsilon^\gamma<r$), 
$$\tilde q_{1t}^\varepsilon(z,t)=(1-\tilde \mu^j)[\dots]_t+\varepsilon^{1-\gamma}\delta_t^j\mu'\left(\frac{\varphi(z,t)-X_j(t)}{\varepsilon^\gamma}\right)[\dots]$$
where the $[\dots]$ is dominated by $\varepsilon^3$. Therefore we have 
$$\|\tilde q_{1t}^\varepsilon\|_{L^\infty}\leq C\varepsilon^3, \|\tilde q_{1t}^\varepsilon\|_{L^1(0;L)}\leq C\varepsilon^3.$$
Similarly, we prove the same estimates for $\tilde q_{1tt}^\varepsilon$. Nevertheless we cannot have such estimates for $\tilde q_{1z}^\varepsilon$. Indeed, we have
$$\tilde q_{1z}^\varepsilon(z,t)=(1-\tilde \mu^j)[\dots]_z-\varepsilon^{-\gamma}\mu'\left(\frac{\varphi(z,t)-X_j(t)}{\varepsilon^\gamma}\right)[\dots]$$
so we just have
$$\|\tilde q_{1z}^\varepsilon\|_{L^\infty}\leq C\varepsilon^{3-\gamma},\|\tilde q_{1z}^\varepsilon\|_{L^1(0;L)}\leq C\varepsilon^3.$$
Similarly, we prove 
$$\|\tilde q_{1zt}^\varepsilon\|_{L^\infty}\leq C\varepsilon^{3-\gamma},\|\tilde q_{1zt}^\varepsilon\|_{L^1(0;L)}\leq C\varepsilon^3,$$
$$\|\tilde q_{1zz}^\varepsilon\|_{L^\infty}\leq C\varepsilon^{3-2\gamma},\|\tilde q_{1zz}^\varepsilon\|_{L^1(0;L)}\leq C\varepsilon^{3-\gamma}.$$

Then, we compute estimates for $\tilde q_2^\varepsilon$:
$$\|\tilde q_{2}^\varepsilon\|_{L^\infty},\|\tilde q_{2t}^\varepsilon\|_{L^\infty},\|\tilde q_{2tt}^\varepsilon\|_{L^\infty}\leq C\varepsilon^2,$$ 
$$\|\tilde q_{2}^\varepsilon\|_{L^1(0;L)}, \|\tilde q_{2t}^\varepsilon\|_{L^1(0;L)}, \|\tilde q_{2tt}^\varepsilon\|_{L^1(0;L)}\leq C\varepsilon^{2+\gamma},$$
$$\|\tilde q_{2z}^\varepsilon\|_{L^\infty},\|\tilde q_{2zt}^\varepsilon\|_{L^\infty}\leq C\varepsilon^{2-\gamma},\|\tilde q_{2z}^\varepsilon\|_{L^1(0;L)},\|\tilde q_{2zt}^\varepsilon\|_{L^1(0;L)}\leq C\varepsilon^2.$$	
$$\|\tilde q_{2zz}^\varepsilon\|_{L^\infty}\leq C\varepsilon^{2-2\gamma},\|\tilde q_{2zz}^\varepsilon\|_{L^1(0;L)}\leq C\varepsilon^{2-2\gamma}.$$	

Eventually, we use matching conditions~\eqref{match0}-\eqref{match1}-\eqref{match2} to prove the estimates on $\tilde q_3^\varepsilon$. Indeed, the properties of viscous shock profiles provide that terms $o(1)$ can be replaced by $e^{-\alpha|\xi|}$ in the matching conditions with $\alpha$ a positive number. So, we have for $z-(j-1)\frac Lm<r$:
$$\begin{array}{rcl}
  I^{j\varepsilon}(\varphi(z,t),t)\hspace{-0.3cm}&=&\hspace{-0.2cm}V^j\left(\frac{z-(j-1)\frac Lm}{\varepsilon},t\right)+\varepsilon V_1^j\left(\frac{z-(j-1)\frac Lm}{\varepsilon},t\right)+\varepsilon^2V_2^j\left(\frac{z-(j-1)\frac Lm}{\varepsilon},t\right)\\
  & = &\hspace{-0.2cm} u^{j\pm}(t)+\varepsilon \left[u_1^{j\pm}(t)+ u_x^{j\pm}(t)\left(\frac{z-(j-1)\frac Lm-\varepsilon\delta_0^j}{\varepsilon}\right)\right]+\varepsilon^2\bigg[u_2^{j\pm}(t)\\
  &&\hspace{-0.2cm}+ u_{1x}^{j\pm}(t)\left(\frac{z-(j-1)\frac Lm-\varepsilon\delta_0^j}{\varepsilon}\right)
  +\frac{1}{2}u_{xx}^{j\pm}(t)\left(\frac{z-(j-1)\frac Lm-\varepsilon\delta_0^j}{\varepsilon}\right)^2\!-u_x^{j\pm}(t)\delta_1^j(t)\bigg]\\
  &&\hspace{-0.2cm}+ \mathcal O\left(e^{-\alpha\frac{|z-(j-1)\frac Lm|}{\varepsilon}}\right)\\
\end{array}$$
and, using Taylor expansion for $u_i(X_j+(z+(j-1)\frac Lm-\varepsilon\delta_0^j))$,
$$\begin{array}{rcl}
  O^\varepsilon(\varphi(z,t),t)\!&=&\!u(z+X_j-(j-1)\frac Lm-\varepsilon\delta^j,t)\\
  &&+\varepsilon u_1(z+X_j-(j-1)\frac Lm-\varepsilon\delta^j,t)\\
  &&+\varepsilon^2u_2(z+X_j-(j-1)\frac Lm-\varepsilon\delta^j,t)\\\\
  &=&\!+u^{j\pm}(t)+u_x^{j\pm}(t)(z-(j-1)\frac Lm+\varepsilon\delta_0^j+\varepsilon^2\delta_1^j)\\
  &&+\frac12  u_{xx}^{j\pm}(t)(z-(j-1)\frac Lm-\varepsilon\delta_0^j)^2+\varepsilon u_1^{j\pm}\\
  &&+\varepsilon u_{1x}^{j\pm}(z-(j-1)\frac Lm-\varepsilon\delta_0^j)+\varepsilon^2u_2^{j\pm}\\
  &&+ \mathcal O(\varepsilon^{3}+(z-(j-1)\frac Lm-\varepsilon\delta_0^j)^{3}).
\end{array}$$
Since 
$$\supp(\tilde q_3^\varepsilon)\subset\left\{(z,t): \varepsilon^\gamma\leq\left|z-(j-1)\frac Lm-\varepsilon \delta^j\right|\leq2\varepsilon^\gamma\right\}, $$
we have 
$$\mathcal O\left(\varepsilon^{3}+\left(z-(j-1)\frac Lm-\varepsilon\delta_0^j\right)^{3}\right)=\mathcal O(\varepsilon^{3\gamma})$$
and we obtain the estimates:
$$\|\tilde q_{3}^\varepsilon\|_{L^\infty},\|\tilde q_{3t}^\varepsilon\|_{L^\infty},\|\tilde q_{3tt}^\varepsilon\|_{L^\infty}\leq C\varepsilon^{2\gamma},$$
$$\|\tilde q_{3}^\varepsilon\|_{L^1(0;L)}, \|\tilde q_{3t}^\varepsilon\|_{L^1(0;L)}, \|\tilde q_{3tt}^\varepsilon\|_{L^1(0;L)}\leq C\varepsilon^{3\gamma},$$
$$\|\tilde q_{3z}^\varepsilon\|_{L^\infty},\|\tilde q_{3zt}^\varepsilon\|_{L^\infty}\leq C\varepsilon^{\gamma},\|\tilde q_{3z}^\varepsilon\|_{L^1(0;L)},\|\tilde q_{3zt}^\varepsilon\|_{L^1(0;L)}\leq C\varepsilon^{2\gamma},$$	
$$\|\tilde q_{3zz}^\varepsilon\|_{L^\infty}\leq C,\|\tilde q_{3zz}^\varepsilon\|_{L^1(0;L)}\leq C\varepsilon^{\gamma}.$$	
This ends the proof of Theorem~\ref{thmuapp}.

\section{Estimates on the Green's function}\label{GREEN}
We now consider the linear operator 
\begin{multline*}
  L^\varepsilon \tilde w=\tilde w_t+\frac1{\varphi_z}\left(\dd f(\tilde u_{app}^\varepsilon)-\varphi_t+\varepsilon\frac{\varphi_{zz}}{\varphi_z^2}\right)\cdot\tilde w_z-\varepsilon\frac1{\varphi_z^2}\tilde w_{zz}\\
  +\left(\left(\frac{1}{\varphi_z}\left(\dd f(\tilde u_{app}^\varepsilon)-\varphi_t+\varepsilon\frac{\varphi_{zz}}{\varphi_z^2}\right)\right)_z-\dd g(\tilde u_{app}^\varepsilon)\right) \cdot \tilde w.
\end{multline*}
The aim of this section is to prove the following theorem:
\begin{theorem}\label{thmGreen}
  There exists a Green's function $G^\varepsilon(t,\tau, z, y)$ of the linear operator $L^\varepsilon$ defined for $0\leq\tau,$ $t\leq T^*,z,y\in\mathbb R$ such that $G^\varepsilon(t,\tau,z,y)=0$ if $\tau>t$ and
  \begin{equation}\label{Greenestimate}\sup_{y,\tau\leq T^*}\int_0^{T^*}\int_\mathbb R|G^\varepsilon(t,\tau,z,y)|\dd z\dd t+\sqrt\varepsilon\sup_{y,\tau\leq T^*}\int_0^{T^*}\int_\mathbb R|\partial_zG^\varepsilon(t,\tau,z,y)|\dd z\dd t\leq C\end{equation}
  where $C$ is positive and does not depend on $\varepsilon$.
\end{theorem}	

To find estimates on this Green's function, we use approximations of the Green's function both near the shock and far away from the shocks. 

First, we recall the method of iterative construction of the Green's function of E.~Grenier and F.~Rousset~\cite{GrenierRousset}. Then, to approximate the Green's function near the shocks, we recall the result of K.~Zumbrun and P.~Howard~\cite{ZumbrunHoward} about Green's function for pure viscous profile problem. Far away from the shocks, we use characteristic curves to build some approximate Green's functions. Finally, we combine all these Green's functions to obtain an approximate Green's function of $L^\varepsilon$ and we find bounds on the error terms. 

\subsection{Method}\label{METHODGREEN} 
Here, we recall the method used by E.~Grenier and F.~Rousset in~\cite{GrenierRousset}. We want to construct an approximate Green's function $G_{app}^\varepsilon$ of $L^\varepsilon$ in the form
$$G_{app}^{\varepsilon}(t,\tau,z,y)=\sum_{k=1}^NS_k(t,\tau,z,y)\Pi_k(\tau,y),$$
where $S_k$ are Green's kernels which satisfy~\eqref{Greenestimate} and $\Pi_k\in\mathcal C^\infty([0, T^*]\times\mathbb R,\mathcal L(\mathbb R^n))$ are such that $$\|\Pi_k(t,x)v\|\leq C\|v\|,\quad\forall x\geq0,t\in[0;T^*],v\in\mathbb R^n$$ and $$\sum_{k=1}^N\Pi_k=\Id.$$

We next define the error $R_k(\cdot,\tau,\cdot,y)=L^\varepsilon S_k(\tau,y)$ for $k=1,\dots,N$, and the matrix of errors: $\mathcal M(T_1,T_2)=(\sigma_{kl}(T_1,T_2))_{1\leq k,l\leq N}$ with
$$\sigma_{kl}(T_1,T_2)=\sup_{T_1\leq \tau\leq T_2,y\in \supp\Pi_l}\int_{T_1}^{T_2}\int_\mathbb R|\Pi_k(t,z)R_l(t,\tau,z,y)|\dd z\dd t.$$

Thanks to Theorem 2.2 of~\cite {GrenierRousset}, we just have to prove that there exists $\varepsilon_2>0$ such that $0<T_2-T_1<\varepsilon_2$ implies $$\lim_{p\to\infty}\mathcal M^p(T_1,T_2)=0.$$
\begin{remarkk}
  \em Since we consider the error in a matrix, we need to consider a finite number of Green's kernel $S_k$. We will see in the following that the number of these kernels is proportional to the number of shocks per period. So, this method does not allow us to treat the case of a non-periodic perturbation. 
\end{remarkk}
 
\subsection{Near the shocks}
Near the shock $j$, we can approximate 
$$\frac1{\varphi_z}\left(\dd f(\tilde u_{app}^\varepsilon)-\varphi_t+\varepsilon\frac{\varphi_{zz}}{\varphi_z^2}\right)$$ 
by 
$$\dd f\left(V^j\left(\frac{z-(j-1)\frac Lm}{\varepsilon},\tau\right)\!\right)-X_j'$$ 
and we forget zeroth order term. Therefore, we search the Green's functions for the linear operators
$$L^{\varepsilon j}_\tau w=\partial_t w+\left(\dd f\left(V^j\left(\frac{z-(j-1)\frac Lm}{\varepsilon},\tau\right)\right)-X_j'(\tau)\right)w_z-\varepsilon w_{zz}$$
which depend on $j$ and $\tau<T^*$. As in~\cite{Rousset}, we remark that these Green's functions $G_\tau^{Sj}(t,z,y)$ verify $$G_\tau^{Sj}(t,z,y)=\frac1\varepsilon G_\tau^{HZj}\left(\frac t\varepsilon,\frac{z-(j-1)\frac Lm}{\varepsilon},\frac{y-(j-1)\frac Lm}{\varepsilon}\right)$$ where $G_\tau^{HZj}$ is the Green's function related to the operator 
$$L^j_\tau w=\partial_t w+(\dd f(V^j(z,\tau))-X_j'(\tau))w_z-w_{zz}.$$
In~\cite{ZumbrunHoward}, K.~Zumbrun and P.~Howard obtained estimates on the Green's functions which will be useful to obtain estimates for our operators. Let us denote by $\tilde a_i^{j\pm}(\tau),$ and $ r_i^{j\pm}(\tau)$ the eigenvalues and the associated eigenvectors of $\dd f(u^{j\pm}(\tau))-X_j'(\tau)$.
\begin{proposition}\label{propZumbrunHoward}
  Under hypothesis {\bf(H3)}, we have 
  \begin{equation}
    \begin{split}
      G_\tau^{Sj}\!\left(\!t,z+(j-1)\frac Lm,y+(j-1)\frac Lm\!\right)\!=\!&{\sum_{i,\tilde a_i^{j+}(\tau)>0}\hspace{-0.2cm}\mathcal O\!\left(\frac{\exp\!\left(-\frac{(z-\tilde a_i^{j+}(\tau)t)^2}{M\varepsilon t}\right)\!}{\sqrt{\varepsilon t}}\right)\!r_i^{j+}(\tau)\chi_{z\geq0}}\\
      &\quad+{\hspace{-0.4cm}\sum_{i,\tilde a_i^{j-}(\tau)<0}\hspace{-0.3cm}\mathcal O\!\left(\frac{\exp\!\left(-\frac{(z-\tilde a_i^{j-}(\tau)t)^2}{M\varepsilon t}\right)\!}{\sqrt{\varepsilon t}}\right)\!r_i^{j-}(\tau)\chi_{z\leq0}}\\
      &\qquad+\mathcal O\left(\frac{\exp\left(-\frac{(z-y)^2}{M\varepsilon t}\right)}{\sqrt{\varepsilon t}}e^{-\sigma\frac t\varepsilon}\right)
    \end{split}
  \end{equation}
  \begin{equation}
    \begin{split}
      \partial_zG_\tau^{Sj}\!\left(\!t,z+(j-1)\frac Lm,y+(j-1)\frac Lm\!\right)\!=\!&{\sum_{i,\tilde a_i^{j+}(\tau)>0}\hspace{-0.2cm}\mathcal O\!\left(\frac{\exp\!\left(-\frac{(z-\tilde a_i^{j+}(\tau)t)^2}{M\varepsilon t}\right)\!}{\varepsilon t}\right)\!r_i^{j+}(\tau)\chi_{z\geq0}}\\
      &\ {+\hspace{-0.4cm}\sum_{i,\tilde a_i^{j-}(\tau)<0}\hspace{-0.3cm}\mathcal O\!\left(\frac{\exp\!\left(-\frac{(z-\tilde a_i^{j-}(\tau)t)^2}{M\varepsilon t}\right)\!}{\varepsilon t}\right)\!r_i^{j-}(\tau)\chi_{z\leq0}}\\
      &\quad+\mathcal O\left(\frac{\exp\left(-\frac{(z-y)^2}{M\varepsilon t}\right)}{\varepsilon t}e^{-\sigma\frac t\varepsilon}\right),
    \end{split}
  \end{equation}
  where $M$ and $\sigma$ are positive constants, and $\chi$ designs characteristic function. Moreover, $\mathcal O$'s are at least linear forms, locally bounded in $y$ and uniformly bounded in $(t,z)$.	
\end{proposition}

\subsection{Far away from the shocks}
As in the previous section, we do not search Green's function for $L^\varepsilon$ but for the approximate operator $\tilde L^\varepsilon$ defined by
$$\tilde L^\varepsilon w= w_t+\frac1{\varphi_z}\left(\dd f(u(\varphi(z,t),t))-\varphi_t+\varepsilon\frac{\varphi_{zz}}{\varphi_z^2}\right)\cdot w_z-\frac\varepsilon {\varphi_z^2}w_{zz}.$$
We remark that, as in the previous section, we forget the terms in $w$.

Recall that $\dd f(u(x,t))$ is diagonalizable for all $x,t$, so we can write 
$$\dd f(u(\varphi(z,t),t))=P(u(\varphi(z,t),t))D(u(\varphi(z,t),t))P(u(\varphi(z,t),t))^{-1}$$
with $D(u(\varphi(z,t),t))=\diag(\lambda_i(u(\varphi(z,t),t)))$.

To obtain approximation of the Green's function between two shocks, we define $j$ approximate problems on $\mathbb R$, with continuous solutions. First, we set
$$\lambda_i^j(\varphi(z,t),t)=\left\{
  \begin{array}{ll}
    \lambda_i(u(X_j^+(t),t))&\text{if }z\in\left]-\infty;(j-1)\frac Lm+\varepsilon\delta^j\right],\\
    \lambda_i(u(\varphi(z,t),t))&\text{if }z\in\left](j-1)\frac Lm+\varepsilon\delta^j;j\frac Lm+\varepsilon\delta^{j+1}\right[,\\
    \lambda_i(u(X_{j+1}^-(t),t))&\text{if }z\in\left[j\frac Lm+\varepsilon\delta^{j+1};+\infty\right[.
  \end{array}
\right.$$
Then, we want to find approximate Green's functions for the scalar operators 
$$L_i^jw=w_t+\frac1{\varphi_z}\left(\lambda_i^j-\varphi_t+\varepsilon\frac{\varphi_{zz}}{\varphi_z^2}\right)w_z-\frac\varepsilon{\varphi_z^2}w_{zz}.$$
To do so, we define characteristic curves $\chi_i^j(t,\tau,y)$ by
\begin{equation*}
  \begin{cases}
    &\partial_t\chi_i^j(t,\tau,y)=\lambda_i^j(\chi_i^j(t,\tau,y),t),\quad t\geq\tau, \\
    &\chi_i^j(\tau,\tau,y)=y,    
  \end{cases}
\end{equation*}
and the approximate Green's functions $$G_i^j(t,\tau,z,y)=\frac{\varphi_z(y,\tau)}{\sqrt{4\pi\varepsilon(t-\tau)}}\exp\left(-\frac{(\varphi(z,t)-\chi_i^j(t,\tau,\varphi(y,\tau)))^2}{ 4\varepsilon(t-\tau)}\right).$$
We easily compute the error committed here 
$$L_i^jG_i^j=(\lambda_i^j(\varphi(z,t),t)-\lambda_i^j(\chi_i^j(t,\tau,\varphi(y,\tau)),t))G_{iz}^j(t,\tau,z,y).$$

Before we build the whole Green's function, we introduce some notations.
First, we write $$G^j=\diag(G_i^j).$$
In the sequel, we need to distinguish at each shock the outgoing waves to the incoming waves.  We define 
$$\begin{array}{lcll}
  D^{-in}   	& = & \diag(0,\dots,0,1,\dots,1), 	& \text{with }k-1\text{ unit coefficients,}\\ 
  D^{-out} 	& = & \diag(1,\dots,1,0,\dots,0), &\text{with }k-1\text{ null coefficients,}\\ 
  D^{+in}	& = & \diag(1,\dots,1,0,\dots,0), &\text{with }k\text{ unit coefficients,}\\ 
  D^{+out}	& = & \diag(0,\dots,0,1,\dots,1), &\text{with }k\text{ null coefficients,}
\end{array}$$
so that $D^{\pm in}+D^{\pm out}=\Id$.

Finally, we define the projections 
\begin{eqnarray*}&\mathcal P^{\pm in}(t,z)=P(t,z)D^{\pm in}P(t,z)^{-1},\\&\mathcal P^{\pm out}(t,z)=P(t,z)D^{\pm out}P(t,z)^{-1}.\end{eqnarray*}

\subsection{Approximate Green's function}

Since the shocks are non-characteristic Lax shocks, we have the following inequality on a neighbourhood of each shock $j=1,\dots,m$:
$$|\lambda_i(\tilde u(z,t))-X_j(t)|>C>0 \text{ if } (j-1)\frac Lm-4\eta<z<(j-1)\frac Lm+ 4\eta, i=1,\dots,n.$$ 
We can assume that $\eta$ is such that $4\eta<r$ so that $\alpha_j\equiv 1$ in~(\ref{phi}) on the previous neighbourhood.

Furthermore, we need some cut-off smooth functions
\begin{eqnarray*}
  K^+(z)=\left\{\begin{array}{ll}0\text{ if }z\leq1,\\1\text{ if }z\geq2\end{array}\right.\text{ and }
  K^-(z)=\left\{\begin{array}{ll}1\text{ if }z\leq-2,\\0\text{ if }z\geq-1.\end{array}\right.
\end{eqnarray*}
We also assume the cut-off function $\mu$ already used to read as $\mu=(1-K^+)(1-K^-).$ We can now build an approximate Green's function in the form $$G_{app}^\varepsilon(t,\tau,z,y)=\sum_{j=1}^m\sum_{k=0}^7S_k^j(t,\tau,z,y)\Pi_k^j(\tau,y)$$
where the Green's kernels are periodic with period $(0,0,L,L)$:
$$S_k^j(t,\tau,z,y)=\sum_{l\in\mathbb Z}\tilde S_k^j(t,\tau,z+lL,y+lL)$$
with
\begin{alignat*}{1}
  &\tilde S_0^j(t,\tau,z,y)=\mu\left(\frac {z-(j-1)\frac Lm}{2\eta}\right)\mu\left(\frac {z-(j-1)\frac Lm}{M_3\varepsilon}\right)G_\tau^{Sj}(t-\tau,z,y),\\
  &\tilde S_{1,2}^j(t,\tau,z,y)=\mu\left(\frac {\dots}{2\eta}\right)K^+\left(\frac{z-(j-1)\frac Lm}{M_1\varepsilon}\right)P(t,z)D^{+out}G^j(t,\tau,z,y)P(\tau,y)^{-1},\\
  &\tilde S_3^j(t,\tau,z,y)=\mu\left(\frac {\dots}{2\eta}\right)K^+\left(\frac{z-(j-1)\frac Lm}{M_1\varepsilon}\right)P(t,z)D^{+in}G^j(t,\tau,z,y)P(\tau,y)^{-1},\\
  &\tilde S_4^j(t,\tau,z,y)=\left(\hspace{-0.1cm}K^+\hspace{-0.1cm}\left(\hspace{-0.1cm}\frac{4(z-(j-1)\frac Lm)}{\eta}\hspace{-0.1cm}\right)\hspace{-0.1cm}+K^-\left(\hspace{-0.1cm}\frac{4(z-j\frac Lm)}{\eta}\hspace{-0.1cm}\right)\hspace{-0.1cm}-1\hspace{-0.1cm}\right)P(t,z)G^jP(\tau,y)^{-1},\\
  &\tilde S_5^j(t,\tau,z,y)=\mu\left(\frac {z-j\frac Lm}{2\eta}\right)K^-\left(\frac{z-j\frac Lm}{M_1\varepsilon}\right)P(t,z)D^{-in}G^j(t,\tau,z,y)P(\tau,y)^{-1},\\
  &\tilde S_{6,7}^j(t,\tau,z,y)=\mu\left(\frac {\dots}{2\eta}\right)K^-\left(\frac{z-j\frac Lm}{M_1\varepsilon}\right)P(t,z)D^{-out}G^j(t,\tau,z,y)P(\tau,y)^{-1},
\end{alignat*}
and the projectors are also periodic:
$$\Pi_k^j(\tau,y)=\sum_{l\in\mathbb Z}\tilde \Pi_k^j(\tau,y+lL)$$
with
\begin{alignat*}{1}
  &\tilde \Pi_0^j(\tau,y)=\mu\left(\frac {y-(j-1)\frac Lm}\eta\right)\mu\left(\frac{y-(j-1)\frac Lm}{M_2\varepsilon}\right),\\
  &\tilde \Pi_1^j(\tau,y)=\mu\left(\frac {\dots}\eta\right)K^+\left(\frac{y-(j-1)\frac Lm}{M_2\varepsilon}\right)\hspace{-0.2cm}\left(1-K^+\hspace{-0.1cm}\left(\frac{2(y-(j-1)\frac Lm)}{M_3\varepsilon}\right)\hspace{-0.2cm}\right)\hspace{-0.1cm}\mathcal P^{+out}(\tau,y),\\
  &\tilde \Pi_2^j(\tau,y)=\mu\left(\frac {\dots}\eta\right)K^+\left(\frac{2(y-(j-1)\frac Lm)}{M_3\varepsilon}\right)\mathcal P^{+out}(\tau,y),\\
  &\tilde \Pi_3^j(\tau,y)=\mu\left(\frac {\dots}\eta\right)K^+\left(\frac{y-(j-1)\frac Lm}{M_2\varepsilon}\right)\mathcal P^{+in}(\tau,y),\\
  &\tilde \Pi_4^j(\tau,y)=K^+\left(\frac{y-(j-1)\frac Lm}{\eta}\right)+K^-\left(\frac{y-j\frac Lm}{\eta}\right)-1,\\
  &\tilde \Pi_5^j(\tau,y)=\mu\left(\frac{y-j\frac Lm}\eta\right)K^-\left(\frac{y-j\frac Lm}{M_2\varepsilon}\right)\mathcal P^{-in}(\tau,y),\\
  &\tilde \Pi_6^j(\tau,y)=\mu\left(\frac{\dots}\eta\right)K^-\left(\frac{2(y-j\frac Lm)}{M_3\varepsilon}\right)\mathcal P^{-out}(\tau,y),\\
  &\tilde \Pi_7^j(\tau,y)=\mu\left(\frac{\dots}\eta\right)K^-\left(\frac{y-j\frac Lm}{M_2\varepsilon}\right)\left(1-K^-\left(\frac{2(y-j\frac Lm)}{M_3\varepsilon}\right)\right)\mathcal P^{-out}(\tau,y)).
\end{alignat*}

It appears that all the Green's kernels can be written in the following form:
\begin{equation}\label{TS}S_k^j(t,\tau,z,y)=T(z)S(t,\tau,z,y)\end{equation}
where $T$ is a truncation function. 

We will choose the three constants $M_1, M_2, M_3$ at the end of the estimates on the error matrix so that they verify 
$$4M_1\leq M_2\leq \frac 14M_3.$$
These inequalities are necessary to have $$\sum_{j,k}\Pi_k^j\equiv1$$
and $$G^{app}(\tau,\tau,z,y)=\delta_y(z)\Id.$$

Under these notations, $\tilde S_0^j$ describes the viscous dynamic at the shock $j$, $\tilde S_1^j$ the creation of outgoing waves in a vicinity at the right of the shock $j$, $\tilde S_2^j$ the creation and propagation of outgoing waves away from the shock $j$, at its right, $\tilde S_3^j$ the creation and propagation of incoming waves at the right of the shock $j$. $\tilde S_4^j$ describes the propagation of the waves between the shocks $j$ and $j+1$ Moreover, the kernels $\tilde S_7^j, \tilde S_6^j,\tilde S_5^j$ are the symmetric of respectively $\tilde S_1^j,\tilde S_2^j,\tilde  S_3^j$ for the left of the shock $j+1$. We summarize this splitting in Figure~\ref{schemaS}.
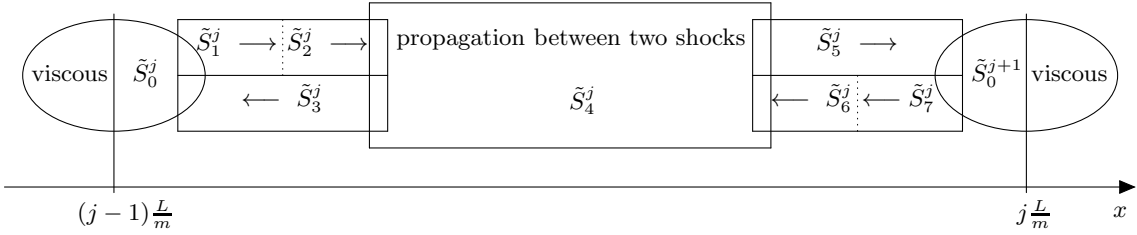
\begin{figure}[ht]
  \begin{center}
    \fontsize{9}{12}\selectfont
    \input{schemaS.tex}
    \caption{Summarize of the splitting by Green's kernels.}
    \label{schemaS}
  \end{center}
\end{figure}

\subsection{Bounds on the error matrix}

As said in Subsection~\ref{METHODGREEN}, to prove Theorem~\ref{thmGreen}, it remains to prove that $\mathcal M^p$ converges to $0$ when $p$ goes to $\infty$. Since the coefficients of $\mathcal M$ are non-negative, it suffices to prove that $\mathcal M$ is bounded above  by an other matrix which has the ``good'' convergence.

To bound the error terms, we use the same method than F.~Rousset in~\cite{Rousset}. We split all the error terms into two parts: the truncation of the error on the kernel and the commutator: $$R_k^j=E_{k1}^j+E_{k2}^j$$
where, with the notation of~(\ref{TS}),
$$E_{k1}^j(t,\tau,z,y)=T(z) L^\varepsilon S(t,\tau,z,y) \text{ and }E_{2k}^j(t,\tau,z,y)=[L^\varepsilon,T(z)]S(t,\tau,z,y).$$

\begin{lemma}\label{lemmaerror}We have the estimates
  \begin{description}
  \item[\emph{at shock $j$:} $\tilde R_0^j$]~\\
    \strut\hspace{-0.8cm}$\|\mathbf1_{|y-(j-1)L/m|\leq 2M_2\varepsilon}E_{01}^j(t,\tau,z,y)\|_{L_{\tau,y}^\infty,L_{t,z}^1}\leq C_1(T+\varepsilon)$,\\
    \strut\hspace{-0.8cm}$\|\mathbf1_{|y-(j-1)L/m|\leq 2M_2\varepsilon}\mathbf1_{\pm z\geq0}\mathcal P_{out}^\pm E_{02}^j(t,\tau,z,y)\|_{L_{\tau,y}^\infty,L_{t,z}^1}\leq C_2$,\\
    \strut\hspace{-0.8cm}$\|\mathbf1_{|y-(j-1)L/m|\leq 2M_2\varepsilon}\mathbf1_{\pm z\geq0}\mathcal P_{in}^\pm E_{02}^j(t,\tau,z,y)\|_{L_{\tau,y}^\infty,L_{t,z}^1}\leq C_3+C_2T$;\\
    \raggedright\item[\emph{for the outgoing waves:} $\tilde R_1^j,\tilde R_2^j,\tilde R_6^j,\tilde R_7^j$]. Let $M\geq M_2$. we have: \\
    \strut\hspace{-0.8cm}$\|\mathbf1_{y-(j-1)L/m\geq M\varepsilon}E_{11}^j(t,\tau,z,y)\|_{L_{\tau,y}^\infty,L_{t,z}^1},\|\mathbf1_{y\geq M\varepsilon}E_{21}^j(t,\tau,z,y)\|_{L_{\tau,y}^\infty,L_{t,z}^1}\!$\\\raggedleft{$\leq\! C_4(T+\varepsilon^{2\gamma-1})+C_5$,}\\
    \raggedright\strut\hspace{-0.8cm}$\|\mathbf1_{y-(j-1)L/m\leq-M\varepsilon}E_{61}^j(t,\tau,z,y)\|_{L_{\tau,y}^\infty,L_{t,z}^1},\|\mathbf1_{y\leq-M\varepsilon}E_{71}^j(t,\tau,z,y)\|_{L_{\tau,y}^\infty,L_{t,z}^1}$\\\raggedleft$\leq\! C_4(T+\varepsilon^{2\gamma-1})+C_5$,\\
    \raggedright\strut\hspace{-0.8cm}$\|\mathbf1_{y-(j-1)L/m\geq M\varepsilon}E_{12}^j(t,\tau,z,y)\|_{L_{\tau,y}^\infty,L_{t,z}^1},\|\mathbf1_{y\geq M\varepsilon}E_{22}^j(t,\tau,z,y)\|_{L_{\tau,y}^\infty,L_{t,z}^1}\!\leq \!C_5+C_{1}(T+\varepsilon)$,\\
    \raggedright\strut\hspace{-0.8cm}$\|\mathbf1_{y-(j-1)L/m\leq- M\varepsilon}E_{62}^j(t,\tau,z,y)\|_{L_{\tau,y}^\infty,L_{t,z}^1},\|\mathbf1_{y\leq- M\varepsilon}E_{72}^j(t,\tau,z,y)\|_{L_{\tau,y}^\infty,L_{t,z}^1}$\\\raggedleft$\leq\! C_5+C_{1}(T+\varepsilon)$;\\
  \raggedright\item[\emph{for the incoming waves:} $\tilde R_3^j,\tilde R_5^j$]~\\
    \strut\hspace{-0.8cm}$\|\mathbf1_{y-(j-1)L/m\geq M\varepsilon}E_{31}^j(t,\tau,z,y)\|_{L_{\tau,y}^\infty,L_{t,z}^1},\|\mathbf1_{y\leq-M\varepsilon}E_{51}^j(t,\tau,z,y)\|_{L_{\tau,y}^\infty,L_{t,z}^1}$\\\raggedleft$\leq C_6(T+\varepsilon^{2\gamma-1})+C_7$,\\
    \raggedright\strut\hspace{-0.8cm}$\|\mathbf1_{y-(j-1)L/m\geq M\varepsilon}E_{32}^j(t,\tau,z,y)\|_{L_{\tau,y}^\infty,L_{t,z}^1},\|\mathbf1_{y\leq- M\varepsilon}E_{52}^j(t,\tau,z,y)\|_{L_{\tau,y}^\infty,L_{t,z}^1}$\\\raggedleft$\leq C_8+C_{1}(T+\varepsilon)$;\\
  \raggedright\item[\emph{between two shocks:} $\tilde R_4^j$]~\\
    \strut\hspace{-0.8cm}$\|R_{4}^j\|_{L_{\tau,y}^\infty,L_{t,z}^1},\leq C_8(T+\varepsilon) $,
  \end{description}
  where $C_1$ is locally bounded in $M_2,M_3$,$C_2$ is locally bounded in $M_2$ uniformly in $M_3$, $C_3$ depends only on $M_2$ and $M_3$ and goes to $0$ as $M_3\to+\infty$, $C_4$ is independent of $M_1,M_2$ and $M_3$, $C_5$ depends only on $M$ and goes to $0$ as $M\to+\infty$, $C_6$ is locally bounded in $M_1$, $C_7$ goes to $0$ as $M_1\to+\infty$, and $C_8$ is bounded uniformly in $M_1$.
\end{lemma}

\begin{proof}
  We do not give here the complete proof of the lemma. Mainly, it deals with terms that are not treated in~\cite{Rousset}: zeroth order terms and terms related to $\varphi$ or $\eta.$

  First, we consider the error at the shock $j$. On the support of $E_{01}^j$, we have that $\varphi_{z}(\cdot,t)\equiv1$ and $\varphi_{t}(\cdot,t)\equiv X_{j}'(t)-\varepsilon\delta_{t}^j(t)$ so we obtain
  \begin{multline*}
    E_{01}^j=\mu\left(\frac {z-(j-1)\frac Lm}{2\eta}\right)\mu\left(\frac {z-(j-1)\frac Lm}{M_3\varepsilon}\right)\Bigg[h\ G_\tau^{Sj}(t-\tau,z,y)\\
    +\left(\dd f(\tilde u_{app}^\varepsilon) -\dd f\left(V^j\left(\frac{z-(j-1)/m}{\varepsilon},\tau\right)\right)+X_j'(\tau)-X_j'(t)+\varepsilon\delta_{t}^j\right) G_{\tau z}^{Sj}\Bigg]
  \end{multline*}
  where $h$ is a bounded function of $t,z$.
  Hence, using the fact that $\dd f$ and $X_j$ are smooth and the expression of $u_{app}^\varepsilon$, we have $$|E_{01}^j|\leq C\mu\left(\frac {z-(j-1)\frac Lm}{2\eta}\right)\mu\left(\frac {z-(j-1)\frac Lm}{M_3\varepsilon}\right)\left[(|t-\tau|+\varepsilon)|G_{\tau z}^{Sj}|+|G_\tau^{Sj}|\right]$$
where $C$ is locally bounded in $M_3$.
  The calculations of~\cite{Rousset} give directly a bound for the first term, in $G_{\tau z}^{Sj}$. So, it only remains to bound the integral:
  $$\int_\tau^T\int_{(j-1)L/m-2M_3\varepsilon}^{(j-1)L/m+2M_3\varepsilon}|G_\tau^{Sj}(t-\tau,z,y)|\dd z\dd t,$$
  and, thanks to Proposition~\ref{propZumbrunHoward}, to estimate:
  $$\int_\tau^T\int_{-M_3\varepsilon}^{M_3\varepsilon}\frac 1{\sqrt{\varepsilon(t-\tau)}}\exp\left(-\frac{(z-a(t-\tau))^2}{M\varepsilon(t-\tau)}\right)\dd z\dd t.$$
  Using the classical change of variable $z'=\frac{z-a(t-\tau)}{\sqrt{M\varepsilon(t-\tau)}}$, we obtain the bound $$\int_\tau^T\int_{(j-1)L/m-2M_3\varepsilon}^{(j-1)L/m+2M_3\varepsilon}|G_\tau^{Sj}(t-\tau,z,y)|\dd z\dd t\leq CT$$ so this concludes the estimate for $E_{01}^j$. For $E_{02}^j$, the linear term in $G_\tau^{Sj}$ disappears in the commutator. Hence, there is no change with the proof of F.~Rousset.
  
  For the estimates on $\tilde R_{i}^{j},i=1,2,3,5,6,$ and $7$, we first remark that $\varphi_{z}\equiv1$ on the support of the errors, so $E_{i1}^j$ is bounded as in~\cite{Rousset}, except for the zeroth order term which is treated as in the case of $E_{01}^j$. So, we only consider the estimate on $E_{i2}^j$. We first remark that the support of this error is not of size $\varepsilon$. Indeed, the truncation $\mu\left(\frac{z-(j-1)\frac Lm}{2\eta}\right)$ adds some error terms with support of size $4\eta$. Though, we have to bound: 
  $$\begin{array}{rl}
    \displaystyle\int_{\tau}^T\int_{(j-1)\frac Lm+M_{1}\varepsilon}^{(j-1)\frac Lm+4\eta}& \displaystyle (\dd f(\tilde u_{app}^\varepsilon)-X_{j}'+\varepsilon{\delta_t^j})\frac1{2\eta}\mu'K^+PD^{+out}G^jP^{-1}\\
    &\displaystyle-\varepsilon\left[\frac1{4\eta^2}\mu''K^+PD^{+out}G^jP^{-1} + \frac1{2\eta} \mu'K^+(PD^{+out}G^j)_{z}P^{-1}\right]\dd z\dd t.
    \end{array}$$
  The two terms in $G^j$ are bounded by $CT$, and the term in $(PD^{+out}G^j)_z$ is bounded by $C(\varepsilon+T)$. Indeed, for $1\leq i\leq n,$
  \begin{multline*} 
    \int_{\tau}^T\int_{(j-1)\frac Lm+M_{1}\varepsilon}^{(j-1)\frac Lm+4\eta}|G_i^j| \\= \int_{\tau}^T\int_{(j-1)\frac Lm+M_{1}\varepsilon}^{(j-1)\frac Lm+4\eta}\frac{|\varphi_z(y,\tau)|}{\sqrt{4\pi\varepsilon(t-\tau)}}\exp\left(-\frac{(\varphi(z,t)-\chi_i^j(t,\tau,\varphi(y,\tau)))^2}{4\varepsilon(t-\tau)}\right)\dd z\dd t.
  \end{multline*}
  Hence, using the change of variable $z'=\frac{\varphi(z,t)-\chi_{i}^j(t,\tau,\varphi(y,\tau))}{\sqrt{4\varepsilon(t-\tau)}}$, we obtain
  $$\int_{\tau}^T\int_{(j-1)\frac Lm+M_{1}\varepsilon}^{(j-1)\frac Lm+4\eta}|G_i^j|\leq C\int_{\tau}^T\int_{\mathbb R} e^{-z'^2}\dd z'\dd t\leq CT.$$
  Similarly, we have
  \begin{multline*}
    \int_{\tau}^T\int_{(j-1)\frac Lm+M_{1}\varepsilon}^{(j-1)\frac Lm+4\eta}|G_{iz}^j|\\=\int_{\tau}^T\int_{(j-1)\frac Lm+M_{1}\varepsilon}^{(j-1)\frac Lm+4\eta}\frac{|\varphi_{z}(y,\tau)|}{\sqrt{4\pi\varepsilon(t-\tau)}}\exp\left(-\frac{(\varphi-\chi_{i}^j)^2}{4\varepsilon(t-\tau)}\right)\left(-\varphi_{z}\frac{\varphi-\chi_{i}^j}{2\varepsilon(t-\tau)}\right)\dd z\dd t.
  \end{multline*}
  And with the same change of variable, we obtain:
  $$\varepsilon\int_{\tau}^T\int_{(j-1)\frac Lm+M_{1}\varepsilon}^{(j-1)\frac Lm+4\eta} |G_{iz}^j|\leq C\sqrt\varepsilon \int_{\tau}^T\frac1{\sqrt{t-\tau}}\int_{\mathbb R} |z'|e^{-z'^2}\dd z'\dd t\leq C(T+\varepsilon),$$
  which gives the estimate for $E_{i2}^{j}.$
  
  Since there is not new difficulty in the proof of the estimates on $R_{4}^j$, we do not develop it here.
\end{proof}

We now use Lemma~\ref{lemmaerror} to bound the matrix $\mathcal{M}$. Since the Green's kernel depends on the shock, we note 
$$\sigma_{kl}^{ij}(T_1,T_2)=\sup_{T_1\leq \tau\leq T_2,y\in \supp\Pi_l^j}\int_{T_1}^{T_2}\int_\mathbb R|\Pi_k^i(t,z)R_l^j(t,\tau,z,y)|\dd z\dd t.$$
Since two shocks do not interact, the error coefficients $\sigma_{kl}^{ij}$ vanish for $i\neq j$ and $kl\neq0$ and for $|i-j|>1$. So the error matrix $\mathcal M$ is bounded:
$$\mathcal M\leq \left(\begin{array}{cccccc}	
    M_1 		& M_2 		&  		&	 	&  		& M_3 \\
    M_3 		& M_1	 	& M_2 	& 		&	 	& \\
    & \ddots 		& \ddots 	& \ddots 	& 		&\\
    &  			& \ddots	& \ddots	& \ddots	&\\
    &  			&  		& \ddots	& \ddots	& M_2\\
    M_2 		&  			& 		& 		& M_3 	& M_1
  \end{array}\right)$$
where $M_2$ is null except on the first column, and $M_3$ is null except on the first line. Moreover, using Lemma~\ref{lemmaerror}, and it was done in~\cite{Rousset}, we can choose $\alpha<1/2$ and $M_1,M_2$ such that when $M_3\to+\infty,\varepsilon,T\to0$, the matrices tend to
$$M_1\to\left(\begin{array}{cccccccc}
    \cdot	& \alpha 	& \cdot 	& C 		& \cdot	& \cdot	& \cdot	& \cdot\\
    \cdot	& \alpha 	& \cdot	& \alpha 	& \cdot	& \cdot	& \cdot	& \cdot\\
    C 	& \alpha 	& \cdot	& \alpha 	& \cdot 	& \cdot	& \cdot	& \cdot\\
    \cdot	& \alpha 	& \cdot	& \alpha 	& \cdot 	& \cdot	& \cdot	& \cdot\\
    \cdot	& \alpha 	& \cdot	& C		& \cdot	& C 		& \cdot	& \alpha\\
    \cdot	& \cdot	& \cdot	& \cdot 	& \cdot 	& \alpha 	& \cdot	& \alpha\\
    \cdot	& \cdot	& \cdot	& \cdot	& \cdot	& \alpha 	& \cdot	& \alpha\\
    \cdot	& \cdot	& \cdot	& \cdot	& \cdot 	& \alpha 	& \cdot	& \alpha\\
  \end{array}\right)
,\quad M_2\to\left(\begin{array}{cccccccc}
    \cdot 	& \cdot & \cdot & \cdot & \cdot & \cdot & \cdot & \cdot\\
    \cdot 	& \cdot & \cdot & \cdot & \cdot & \cdot & \cdot & \cdot\\
    \cdot 	& \cdot & \cdot & \cdot & \cdot & \cdot & \cdot & \cdot\\
    \cdot 	& \cdot & \cdot & \cdot & \cdot & \cdot & \cdot & \cdot\\
    \cdot 	& \cdot & \cdot & \cdot & \cdot & \cdot & \cdot & \cdot\\
    \cdot 	& \cdot & \cdot & \cdot & \cdot & \cdot & \cdot & \cdot\\
    C 		& \cdot & \cdot & \cdot & \cdot & \cdot & \cdot & \cdot\\
    \cdot 	& \cdot & \cdot & \cdot & \cdot & \cdot & \cdot & \cdot	
  \end{array}\right),$$
$$M_3\to\left(\begin{array}{cccccccc}
    \cdot & \cdot & \cdot & \cdot & \cdot & C 		& \cdot & \alpha\\
    \cdot & \cdot & \cdot & \cdot & \cdot & \cdot 	& \cdot & \cdot\\
    \cdot & \cdot & \cdot & \cdot & \cdot & \cdot 	& \cdot & \cdot\\
    \cdot & \cdot & \cdot & \cdot & \cdot & \cdot 	& \cdot & \cdot\\
    \cdot & \cdot & \cdot & \cdot & \cdot & \cdot 	& \cdot & \cdot\\
    \cdot & \cdot & \cdot & \cdot & \cdot & \cdot 	& \cdot & \cdot\\
    \cdot & \cdot & \cdot & \cdot & \cdot & \cdot 	& \cdot & \cdot\\
    \cdot & \cdot & \cdot & \cdot & \cdot & \cdot 	& \cdot & \cdot	
  \end{array}\right).$$

So that $\mathcal M\to\tilde{\mathcal M}$ where the eigenvalues of $\tilde {\mathcal M}$ are $0$ and $2\alpha$. 
We can now conclude that we have $\mathcal M^p\to0$ when $p\to\infty$. This ends the proof of Theorem~\ref{thmGreen}.

\section{Convergence}\label{CONVERGENCEPERSISTENCE}

The purpose of this section is to prove Theorem~\ref{thmconvergence} and to conclude the proof of Theorem~\ref{thethmpersistence}. For this sake, it remains to show that solution of
\begin{eqnarray}
  \label{Leps}&&L^\varepsilon\tilde w=-\tilde q^\varepsilon+Q_1(\tilde u_{app}^\varepsilon,\tilde w)-\frac1{\varphi_z}Q_2(\tilde u_{app}^\varepsilon,\tilde w)_z,\\
  \label{w0}&&\tilde w(z,0)=0
\end{eqnarray}
vanishes as $\varepsilon\to0$. In the previous section, we obtain estimate~(\ref{Greenestimate}) on the Green's function of operator $L^\varepsilon$. We recall that $Q_1$ and $Q_2$ are at least quadratic in $\tilde w$ and $\tilde q^\varepsilon$ verifies inequalities~\eqref{Linfinity},~\eqref{L1},~\eqref{qz}, and~(\ref{qzz}).

As in~\cite{GoodmanXin},~\cite{GrenierRousset} and~\cite{Rousset}, we use standard arguments for parabolic equations. First, we remark that local existence of a smooth solution $\tilde w$ for~\eqref{Leps}-\eqref{w0} is classical. 
Then we define $$T^\varepsilon=\sup\{T_1\in[0;T^*],\exists \tilde w\text{ solution on }\mathbb R\times [0;T_1),E(T_1)\leq1\},$$
where 
\begin{multline*}
  E(T_1)=\int_0^{T_1}\int_0^L\left(\frac{|\tilde w|}{\varepsilon^{3\gamma-\alpha}}+\frac{|\tilde w_{z}|}{\varepsilon^{3\gamma-\alpha-1/2}}+\frac{|\tilde w_t|}{\varepsilon^{3\gamma-2\alpha-1/2}}\right.\\\left.+\frac{|\tilde w_{zz}|}{\varepsilon^{3\gamma-2\alpha-1}}+\frac{|\tilde w_{tz}|}{\varepsilon^{3\gamma-2\alpha-1}}+\frac{|\tilde w_{tzz}|}{\varepsilon^{3\gamma-3\alpha-3/2}}+\frac{|\tilde w_{zzz}|}{\varepsilon^{3\gamma-2\alpha-2}}\right)\dd z\dd t
\end{multline*}
with $\alpha>0$ and $\gamma\in(2/3,1)$, chosen later.

Before estimating the $L^1-$norms, we define the notation:
$$\|\tilde w\|_{1}=\|\tilde w\|_{L^1((0;T^\varepsilon)\times[0,L))},\|\tilde w\|_{\infty}=\|\tilde w\|_{L^\infty((0;T^\varepsilon)\times[0,L))} $$
First, we use Fourier coefficients $c_n$ to estimate $\|\tilde w\|_{L^\infty}$. Indeed,
\begin{equation*}
  \begin{split}
    \tilde w(z,t) &= \int_0^t\tilde w_t(z,s)\dd s = \int_0^t\sum_{n\in\mathbb Z}c_n(\tilde w_t(s))e^{\imath nz}\dd s\\
    &=\int_0^tc_0(\tilde w_t(s))\dd s+\int_0^t\sum_{n\in\mathbb Z\backslash\{0\}}\frac {c_n(\tilde w_{tzz}(s))}{n^2}e^{\imath nz}\dd s
  \end{split}
\end{equation*}
$$\|\tilde w\|_\infty\leq C (\|\tilde w_t\|_1+\|\tilde w_{zzt}\|_1)\leq C\varepsilon^{3\gamma-3\alpha-3/2}$$
when $\varepsilon\leq 1$. Hence, $\|\tilde w\|_\infty$ tends to 0 as $\varepsilon$ goes to 0 if $\gamma$ and $\alpha$ are such that $\gamma-\alpha>1/2.$

We deduce that to prove that the time existence is $T^*$ and that we have the convergence, it remains to prove that $T^\varepsilon=T^*$ for $\varepsilon$ small enough. In the sequel, we suppose that $T^\varepsilon<T^*$ so $E(T^\varepsilon)=1$.

Then, we can use the estimate on $\|\tilde w\|_\infty$ to bound $\tilde w$ and its derivatives. First, using~(\ref{Leps})-(\ref{w0}), we have 
$$\tilde w(t,z)=\int_0^t\int_\mathbb R G^\varepsilon(t,\tau,z,y)\left(-\tilde q^\varepsilon+Q_1(\tilde u_{app}^\varepsilon,\tilde w)-\frac1{\varphi_z}Q_2(\tilde u_{app}^\varepsilon,\tilde w)_z\right)(\tau,y)\dd y\dd\tau.$$
Also, we have for periodic function $\psi$ the estimate $$\int_{\mathbb R}\left(\int_0^LG^\varepsilon(t,\tau,z,y)dz\right)\psi(y)\dd y=\mathcal O\left(\int_0^L\left(\int_0^LG^\varepsilon(t,\tau,z,y)dz\right)\psi(y)\dd y\right).$$
Therefore, we deduce
$$\|\tilde w\|_1\leq C\varepsilon^{3\gamma}+C \|\tilde w\|_\infty \|\tilde w\|_1 + C \|\tilde w\|_\infty(\|\tilde w\|_1+\|\tilde w_z\|_1)$$
so
$$\frac{\|\tilde w\|_1}{\varepsilon^{3\gamma-\alpha}}\leq C(\varepsilon^{\alpha}+\varepsilon^{3\gamma-3\alpha-3/2}(1+\varepsilon^{-1/2})).$$
This can be made smaller than $1$ as $\varepsilon\to0$ if $\alpha$ and $\gamma$ are such that $3\gamma-3\alpha-2>0.$

We now take the $z$ derivative $\tilde w$ and obtain an expression of $\tilde w_z$:
$$\tilde w_z(t,z)=\int_0^t\int_\mathbb R G_z^\varepsilon(t,\tau,z,y)\left(-\tilde q^\varepsilon+Q_1(\tilde u_{app}^\varepsilon,\tilde w)-\frac1{\varphi_z}Q_2(\tilde u_{app}^\varepsilon,\tilde w)_z\right)(\tau,y)\dd y\dd\tau$$
and the estimate:
$$\frac{\|\tilde w_z\|_1}{\varepsilon^{3\gamma-\alpha-1/2}}\leq C(\varepsilon^{\alpha}+\varepsilon^{3\gamma-3\alpha-3/2}(1+\varepsilon^{-1/2})).$$

Differentiating equation~\eqref{Leps} with respect to $t$ or $z$, we obtain equations verified by $w_t$ and $w_z$:
$$L^\varepsilon\tilde w_t=\left(-\tilde q^\varepsilon+Q_1(\tilde u_{app}^\varepsilon,\tilde w)-Q_2(\tilde u_{app}^\varepsilon,\tilde w)_z\right)_t+l_1(\tilde w,\tilde w_z,\varepsilon\tilde w_{zz}),$$
$$  L^\varepsilon\tilde w_z=\left(-\tilde q^\varepsilon+Q_1(\tilde u_{app}^\varepsilon,\tilde w)-Q_2(\tilde u_{app}^\varepsilon,\tilde w)_z\right)_z+l_2(\tilde w,\tilde w_z,\varepsilon\tilde w_{zz})$$
where $l_1$ and $l_2$ are continuous linear forms, uniformly bounded with respect to $\varepsilon$.
Thus, using again the Green's function, we get the inequalities:
$$\frac{\|\tilde w_t\|_1}{\varepsilon^{3\gamma-2\alpha-1/2}}\leq C(\varepsilon^{2\alpha+1/2}+\varepsilon^{3\gamma-3\alpha-3/2}(\varepsilon^ {\alpha+1/2}+1+\varepsilon^\alpha+\varepsilon^{-1/2})+\varepsilon^{\alpha+1/2}+\varepsilon^{\alpha}+\varepsilon^{1/2}),$$
$$\frac{\|\tilde w_{tz}\|_1}{\varepsilon^{3\gamma-2\alpha-1}}\leq C(\varepsilon^{2\alpha+1/2}+\varepsilon^{3\gamma-3\alpha-3/2}(\varepsilon^{\alpha+1/2}+\varepsilon^{-1/2})+\varepsilon^\alpha+\varepsilon^{1/2}),$$
$$\frac{\|\tilde w_{zz}\|_1}{\varepsilon^{3\gamma-2\alpha-1}}\leq C(\varepsilon^{2\alpha+1/2-\gamma}+\varepsilon^{3\gamma-3\alpha-3/2}(\varepsilon^{\alpha+1/2}+\varepsilon^{-1/2})+\varepsilon^\alpha+\varepsilon^{1/2}).$$
In the last inequality, the right-hand side can be made smaller than $1$ as $\varepsilon\to0$ if $\alpha$ and $\gamma$ are such that $2\alpha+1/2-\gamma>0.$

Differentiating again equation~\eqref{Leps} with respect to $t,z$ or $z,z$, we obtain the equalities:
$$L^\varepsilon\tilde w_{tz}=\left(-\tilde q^\varepsilon+Q_1(\tilde u_{app}^\varepsilon,\tilde w)-Q_2(\tilde u_{app}^\varepsilon,\tilde w)_z\right)_{tz}+l_3(\tilde w,\tilde w_z,\tilde w_t,\tilde w_{zz},\tilde w_{zt},\varepsilon \tilde w_{zzz},\varepsilon \tilde w_{zzt}),$$
$$L^\varepsilon\tilde w_{tz}=\left(-\tilde q^\varepsilon+Q_1(\tilde u_{app}^\varepsilon,\tilde w)-Q_2(\tilde u_{app}^\varepsilon,\tilde w)_z\right)_{zz}+l_4(\tilde w,\tilde w_z,\tilde w_{zz},\varepsilon \tilde w_{zzz})$$
where $l_3$ and $l_4$ are continuous linear forms, uniformly bounded with respect to $\varepsilon$.
As seen before, we deduce the inequalities:
$$\frac{\|\tilde w_{tzz}\|_1}{\varepsilon^{3\gamma-3\alpha-3/2}}\leq C(\varepsilon^{3\alpha+1-\gamma}+\varepsilon^{3\gamma-3\alpha-3/2}(\varepsilon^ {2\alpha+1}+\varepsilon^{-1/2})+\varepsilon^{2\alpha+1}+\varepsilon^{\alpha}+\varepsilon^{1/2}),$$
$$\frac{\|\tilde w_{zzz}\|_1}{\varepsilon^{3\gamma-2\alpha-2}}\leq C(\varepsilon^{2\alpha+3/2-2\gamma}+\varepsilon^{3\gamma-3\alpha-3/2}(\varepsilon^{\alpha+3/2}+\varepsilon^{-1/2})+\varepsilon^{\alpha+3/2}+\varepsilon^{1/2}).$$
The both bounds can be made smaller than $1$ as $\varepsilon\to0$ if $\alpha$ and $\gamma$ are such that $3\alpha+1-\gamma>0$ and $2\alpha+3/2-2\gamma>0.$

We can now verify that there exist $\gamma$ and $\alpha$ checking all the previous conditions: they define a non-empty trapeze in the plane $\alpha,\gamma$. Hence, we have proved that for $\varepsilon$ small enough, we have $E(T^\varepsilon)\leq\varepsilon^\beta$ with $\beta>0$. Consequently, we can not have $T^\varepsilon<T^*$ for such an $\varepsilon$. Moreover, using the change of variable $\varphi$, we return to $w=u^\varepsilon-u_{app}^\varepsilon$. Thus, we have
$$\|w\|_\infty\leq C\varepsilon^{3\gamma-3\alpha-3/2},$$
and
$$\|u^\varepsilon-u_{app}^\varepsilon\|_\infty\leq C\varepsilon^{3\gamma-3\alpha-3/2}.$$
Inequality $E(T^*)\leq 1$ also gives $$\|w\|_{L^\infty(L^1)}\leq\|w_t\|_1\leq C\|\tilde w_t\|_1\leq C\varepsilon^{3\gamma-\alpha_1/2}$$
so $$\|u^\varepsilon-u_{app}^\varepsilon\|_{L^\infty(L^1)}\leq C\varepsilon^{3\gamma-2\alpha-1/2}.$$
This concludes the proof of Theorem~\ref{thmconvergence}.

It only remains to prove Theorem~\ref{thethmpersistence}. Since $\|u_{app}^\varepsilon-u\|_{L^\infty(L^1)}\to0,$ we have the convergence in $L^{\infty}((0;T^*),L^1(0;L)).$ Moreover, the fast convergences of the viscous shock profiles $V^j$ give the last point of the theorem. 

\section{Conclusion and perspectives}
In this article, we have proved the persistence of solutions of the inviscid equation~\eqref{eqhyp} close to roll-waves by adding full viscosity. One of the main assumptions that we have taken is the periodicity of the solution of~\eqref{eqhyp}. This one is not necessary in the construction of the approximate solution $u_{app}^\varepsilon$, but it gives that $u_1$ and $u_2$ stay bounded (because periodic). An idea to weaken this assumption would be that solution $u$ of~\eqref{eqhyp} approximates the roll-wave as $|x|$ goes to infinity (in particular, the shock curves would be closer as $|x|$ goes to infinity). Moreover, the periodicity of $u$ allows us to use the method of \cite{GrenierRousset} to construct the Green's function of $L^\varepsilon$. Indeed, in this step, the number of Green's functions that we consider is proportional to the number of shocks. In the periodic case, taking into account the repetitions, it returns to a finite number of periodic Green's kernels. Thus, we can write a matrix of errors, and deduce Theorem \ref{thmGreen} on the existence of the Green's function relative to $L^\varepsilon$ and estimates on this Green's function. Another way to hope to obtain a finite number of Green's functions could be to assume that $u$ coincides with the roll-wave outside a bounded domain.

Furthermore, Theorem~\ref{thethmpersistence} proved here is valid in the case of an artificial viscosity. Therefore, one should also study the persistence in the case of real viscosity as presented to the system of Saint Venant~\eqref{StVenant}.

Finally, one can also be interested in what happens in the multidimensional case. For this, we could build on work already done in the case of a single multidimensional shock, based on the study of Evans' functions at each shock~\cite{GuesMetivierWilliamsZumbrun}.

\paragraph*{Acknowledgment.} The author wishes to thank Pascal Noble for suggesting the problem and for fruitful discussions.
\bibliographystyle{plain}
\bibliography{biblio}
\end{document}

%% file: u.tex
\begin{tikzpicture}[line cap=round,line join=round,>=triangle 45,x=3.0cm,y=1.4cm]
\clip(-0.5,-2.4) rectangle (4.5,1.6);

\draw[->] (-0.5,-1.6) -- (4.5,-1.6);
\draw (4.55,-2) node [anchor=south east] {$x$};
\draw[->] (-0.3,-2.4) -- (-0.3,1.6);
\draw (-0.45,1.6) node [anchor=north west] {$u$};

\draw[line width=0.2pt, dash pattern=on 1pt off 1pt] (0,-1) .. controls (0.48,-0.79) and (0.64,0.81) .. (1,1.16);
\draw[line width=0.2pt, dash pattern=on 1pt off 1pt] (1,-1) .. controls (1.48,-0.79) and (1.64,0.81) .. (2,1.16);
\draw[line width=0.2pt, dash pattern=on 1pt off 1pt] (2,-1) .. controls (2.48,-0.79) and (2.64,0.81) .. (3,1.16);
\draw[line width=0.2pt, dash pattern=on 1pt off 1pt] (3,-1) .. controls (3.48,-0.79) and (3.64,0.81) .. (4,1.16);
\draw [line width=0.2pt,dash pattern=on 1pt off 1pt] (0,1.26)-- (0,-1.1);\draw [line width=0.2pt,dash pattern=on 1pt off 1pt] (0,-1.1)-- (0,-1.65);
\draw (-0.05,-2) node [anchor=south west] {$ct$};
\draw [line width=0.2pt,dash pattern=on 1pt off 1pt] (1,1.26)-- (1,-1.1);\draw [line width=0.2pt,dash pattern=on 1pt off 1pt] (1,-1.1)-- (1,-1.65);
\draw (0.65,-2.1) node [anchor=south west] {$\frac Lm+ct$};
\draw [line width=0.2pt,dash pattern=on 1pt off 1pt] (2,1.26)-- (2,-1.1);\draw [line width=0.2pt,dash pattern=on 1pt off 1pt] (2,-1.1)-- (2,-1.65);
\draw (1.9,-2.1) node [anchor=south west] {$\frac {2L}m+ct$};
\draw [line width=0.2pt,dash pattern=on 1pt off 1pt] (3,1.26)-- (3,-1.1);\draw [line width=0.2pt,dash pattern=on 1pt off 1pt] (3,-1.1)-- (3,-1.65);
\draw (2.9,-2.1) node [anchor=south west] {$\frac {3L}m+ct$};
\draw [line width=0.2pt,dash pattern=on 1pt off 1pt] (4,1.26)-- (4,-1.1);\draw [line width=0.2pt,dash pattern=on 1pt off 1pt] (4,-1.1)-- (4,-1.65);
\draw (3.9,-2.1) node [anchor=south west] {$\frac {4L}m+ct$};

\draw[line width=0.6pt] (-0.09,-1.15) .. controls (0.43,-0.93) and (0.57,0.91) .. (1.12,1.09);
\draw[line width=0.6pt] (1.12,-1.35) .. controls (1.45,-0.8) and (1.56,0.94) .. (1.9,1.23);
\draw[line width=0.6pt] (1.9,-1) .. controls (2.38,-0.84) and (2.58,0.97) .. (2.9,1.45);
\draw[line width=0.6pt] (2.9,-1.3) .. controls (3.4,-0.98) and (3.52,0.82) .. (3.91,1);
\draw (-0.09,1.1)-- (-0.09,-1.25);\draw (-0.09,-1.25) -- (-0.09,-2);
\draw (-0.25,-2.4) node [anchor=south west] {$X_1(t)$};
\draw (1.12,1.19)-- (1.12,-1.45);\draw (1.12,-1.45) -- (1.12,-2);
\draw (0.97,-2.4) node [anchor=south west] {$X_2(t)$};
\draw (1.9,1.33)-- (1.9,-1.1);\draw (1.9,-1.1) -- (1.9,-2);
\draw (1.75,-2.4) node [anchor=south west] {$X_3(t)$};
\draw (2.9,1.55)-- (2.9,-1.4);\draw (2.9,-1.4) -- (2.9,-2);
\draw (2.75,-2.4) node [anchor=south west] {$X_4(t)$};
\draw (3.91,1.1)-- (3.91,-1.25);\draw (3.91,-1.25) -- (3.91,-2);
\draw (3.61,-2.4) node [anchor=south west] {$X_1(t)+L$};

\draw[line width=0.4pt] (3.91,-1.15) .. controls (4.43,-0.93) and (4.57,0.91) .. (5.12,1.09);
\draw[line width=0.4pt] (-1.1,-1.3) .. controls (-0.6,-0.98) and (-0.48,0.82) .. (-0.09,1);

\end{tikzpicture}

%% file: z.tex
\begin{tikzpicture}[line cap=round,line join=round,>=triangle 45,x=1.4cm,y=1.4cm]
\clip(-0.4,-0.3) rectangle (4.5,4.5);

\draw[->] (-0.3,0) -- (4.5,0);
\draw[->] (0,-0.3) -- (0,4.5);
\foreach \y in {,1,2,3,4}
\draw[shift={(0,\y)}] (2pt,0pt) -- (-2pt,0pt);

\draw (0.6,1)-- (1.5,2);
\draw (1.5,2)-- (3.09,3);
\draw (3.09,3)-- (4,4);
\draw (0,0)-- (0.6,1);

\draw [dotted] (0,1)-- (4,1);
\draw [dotted] (0,2)-- (4,2);
\draw [dotted] (0,3)-- (4,3);
\draw [dotted] (0,4)-- (4,4);

\draw [dotted] (0.6,0)-- (0.6,4); 
\draw [dotted] (1.5,0)-- (1.5,4);
\draw [dotted] (3.09,0)-- (3.09,4);
\draw [dotted] (4,0)-- (4,4);

\draw (-0.1,0.06) node[anchor=north west] {$X_1$};
\draw (0.4,0.06) node[anchor=north west] {$X_2$};
\draw (1.3,0.06) node[anchor=north west] {$X_3$};
\draw (2.89,0.06) node[anchor=north west] {$X_4$};
\draw (3.45,0.06) node[anchor=north west] {$X_1+L$};
\draw (0,1.24) node[anchor=north east] {$ \frac{ L }{ m } $};
\draw (0,2.24) node[anchor=north east] {$ \frac{ 2L }{ m } $};
\draw (0,3.24) node[anchor=north east] {$ \frac{ 3L }{ m } $};
\draw (-0.05,4.17) node[anchor=north east] {$L$};
\end{tikzpicture}

%% file: w1.tex
\begin{tikzpicture}[line cap=round,line join=round,x=5.0cm,y=4.0cm]
\clip(-0.5,-0.15) rectangle (2.5,1.25);

\draw[line width=0.6pt] (0,0)-- (2,0);
\draw [-open triangle 45, line width=0.2pt] (-0.5,0)--(2.5,0);
\draw (2.5,-0.1) node[anchor=south east] {$x$};

\draw[-triangle 45,line width=0.6pt] (0,0)-- (0,1.2);
\draw (-0.1,1.25) node[anchor=north west] {$t$};

\draw[line width=0.6pt] (1,0)-- (1,1.1);
\draw[line width=0.6pt] (2,0)-- (2,1.1);

\draw[line width=1pt,dotted] (0.37,0) .. controls (0.5211, 0.52591) and (0.79606, 0.87745) .. coordinate[midway](flechedot1) (1,1);
\draw[-angle 60,line width=1pt] (flechedot1) -- +(.00007,.0001);
\draw[line width=1pt,dotted] (1.22,0) .. controls (1.73594, 0.15285) and (1.64453, 0.85593) .. coordinate[midway](flechedot2) (2,1);
\draw[-angle 60,line width=1pt] (flechedot2) -- +(.00005,.0001);

\draw[line width=0.8pt,dash pattern=on 1pt off 3pt on 5pt off 3pt] (0,1) .. controls (0.14824, 0.63353) and (0.19677, 0.21742) .. coordinate[midway](flechepat1) (0.25499, 0);
\draw[-angle 60,line width=1pt] (flechepat1) -- +(-.00003,.0001);
\draw[line width=0.8pt,dash pattern=on 1pt off 3pt on 5pt off 3pt] (1,1) .. controls (1.24375, 0.3896) and (1.63047, 0.39678) .. coordinate[midway](flechepat2)(1.71,0);
\draw[-angle 60,line width=1pt] (flechepat2) -- +(-.0001,.0001);

\draw[line width=0.6pt] (0,1) .. controls (0.19353, 0.79136) and (0.31322, 0.29634) .. coordinate[midway](fleche1)(0.43,0);
\draw[-angle 60,line width=1pt] (fleche1) -- +(-.00005,.0001);
\draw[line width=0.6pt] (0.5,0) .. controls (0.61094, 0.24612) and (0.89922, 0.31786) .. coordinate[midway](fleche2)(1,1);
\draw[-angle 60,line width=1pt] (fleche2) -- +(.00008,.0001);
\draw[line width=0.6pt] (1,1) .. controls (1.04515, 0.72679) and (1.14866, 0.17438) .. coordinate[midway](fleche3)(1.36,0);
\draw[-angle 60,line width=1pt] (fleche3) -- +(-.000035,.0001);
\draw[line width=0.6pt] (1.62,0) .. controls (1.72188, 0.84875) and (1.91172, 0.55461) .. coordinate[midway](fleche4) (2,1);
\draw[-angle 60,line width=1pt] (fleche4) -- +(.00008,.0001);

\draw[line width=0.2pt] (-0.38,0) .. controls (-0.27812, 0.84875) and (-0.08828, 0.55461) .. (0,1);
\draw[line width=0.4pt,dotted] (-0.78,0) .. controls (-0.26406, 0.15285) and (-0.35547, 0.85593) .. (0,1);
\draw[line width=0.2pt] (2,1) .. controls (2.19353, 0.79136) and (2.31322, 0.29634) .. (2.43,0);
\draw[line width=0.3pt,dash pattern=on 1pt off 3pt on 5pt off 3pt] (2,1) .. controls (2.14824, 0.63353) and (2.19677, 0.21742) .. (2.25499, 0);

\draw (-0.05,0) node[anchor=north west] {$0$};
\draw (0.93,-0.) node[anchor=north west] {$L/2$};
\draw (1.95,-0.) node[anchor=north west] {$L$};
\draw (0.5,0.80) node[anchor=north west] {$i>k$};
\draw (0.67,0.29) node[anchor=north west] {$i=k$};
\draw (1.17,0.8) node[anchor=north west] {$i<k$};
\draw (1.01,0.29) node[anchor=north west] {$i=k$};

\draw [line width=0.6pt] (0.465,-0.02) -- (0.465,0.02);
\draw (0.415,0) node[anchor=north west] {$s_1$};
\draw [line width=0.6pt] (1.51,-0.02) -- (1.51,0.02);
\draw (1.46,0) node[anchor=north west] {$s_2$};

\end{tikzpicture}

%% file: schemaS.tex
\begin{tikzpicture}[line cap=round,line join=round,>=triangle 45,x=12cm,y=7.4cm]
\clip(-0.12,-0.29) rectangle (1.12,0.135);

\draw[->] (-0.12,-0.2) -- (1.12,-0.2);
\draw (1.12,-0.27) node [anchor=south east] {$x$};

\draw (0,-0.21)-- (0,0.11);
\draw (-0.05,-0.29) node [anchor=south west] {$(j-1)\frac Lm$};
\draw (1,-0.21)-- (1,0.11);
\draw (0.98,-0.29) node [anchor=south west] {$j\frac Lm$};

\draw [rotate around={0:(0,0)}] (0,0) ellipse (0.1 and 0.1);
\draw (-0.1,0.045) node[anchor=north west] {viscous~~ $\tilde S_0^j$};
\draw [rotate around={180:(1,-0)}] (1,-0) ellipse (0.1 and 0.1);
\draw (0.93,0.045) node[anchor=north west] {$\tilde S_0^{j+1}$~viscous};

\draw (0.07,-0.1)-- (0.07,0.1);
\draw (0.07,-0.1)-- (0.3,-0.1);
\draw (0.07,0.1)-- (0.3,0.1);
\draw (0.3,-0.1)-- (0.3,0.1);
\draw (0.07,0)-- (0.3,0);
\draw [line width = 0.3pt,dotted] (0.185,0)-- (0.185,0.1);
\draw (0.08,0.1) node[anchor=north west] {\parbox[0.8cm]{1.5 cm}{$\tilde S_1^j\ \longrightarrow$}};
\draw (0.18,0.1) node[anchor=north west] {\parbox[0.8cm]{1.5 cm}{$\tilde S_2^j\ \longrightarrow$}};
\draw (0.13,0) node[anchor=north west] {\parbox[0.1cm]{2.6 cm}{$\longleftarrow\ \tilde S_3^j$}};

\draw (0.28,-0.13)-- (0.28, 0.13);
\draw (0.28,-0.13)-- (0.72,-0.13);
\draw (0.28, 0.13)-- (0.72, 0.13);
\draw (0.72,-0.13)-- (0.72, 0.13);
\draw (0.3,0.1) node[anchor=north west] {\parbox{4.9 cm}{propagation between two shocks \newline$$\tilde S_4^j$$}};

\draw (0.93,-0.1)-- (0.93,0.1);
\draw (0.93,-0.1)-- (0.7,-0.1);
\draw (0.93,0.1)-- (0.7,0.1);
\draw (0.7,-0.1)-- (0.7,0.1);
\draw (0.93,0)-- (0.7,0);
\draw [line width = 0.3pt,dotted] (0.815,-0.1)-- (0.815,0);
\draw (0.76,0.1) node[anchor=north west] {\parbox[0.1cm]{2.6 cm}{$\tilde S_5^j\ \longrightarrow$}};
\draw (0.71,0) node[anchor=north west] {\parbox[0.8cm]{1.5 cm}{$\longleftarrow\ \tilde S_6^j$}};
\draw (0.81,0) node[anchor=north west] {\parbox[0.8cm]{1.5 cm}{$\longleftarrow\tilde S_7^j$}};

\end{tikzpicture}